\documentclass[a4paper,11pt,oneside]{article}
\usepackage[utf8]{inputenc}
\usepackage[T1]{fontenc}
\usepackage[english]{babel}
\usepackage{amsthm}
\usepackage{array}
\usepackage{tikz}
\usepackage{appendix}
\usepackage{amsmath}
\usepackage{amssymb}
\usepackage{hyperref}
\usepackage{mathrsfs} 
\usepackage{bbm}
\newcommand{\xdownarrow}[1]{%
  {\left\downarrow\vbox to #1{}\right.\kern-\nulldelimiterspace}
}
\hypersetup{
    colorlinks=true, 
    linktoc=all,
    urlcolor=darkgray,   
    linkcolor=blue,  
    citecolor=darkgray,
}
\usepackage{babel}
\usepackage{csquotes}
\definecolor{greenpigment}{rgb}{0.0, 0.65, 0.31}
\definecolor{iceberg}{rgb}{0.44, 0.65, 0.82}

\usepackage{graphicx, sidecap}
\usepackage{algorithm}
\usepackage{algpseudocode}
\usepackage{wasysym}
\usepackage{pifont,stmaryrd}

\usepackage{mathtools}
\usepackage{amsfonts}
\usepackage{multicol}
\usepackage{setspace}
\usepackage{listings}
\usepackage{tikz}
\usetikzlibrary{arrows}
\usetikzlibrary{positioning}
\usepackage{caption}
\usetikzlibrary{shapes,arrows}
\usepackage{multirow}
\usepackage{setspace}
\usepackage[top=1.5cm, bottom=3cm, left=2cm , right=1cm]{geometry}
\usepackage{fancybox}
\usepackage{xcolor, color}
\usepackage{pifont}
\reversemarginpar
\usepackage{url}

\usepackage[all]{xy}
\usepackage{tikz-cd}

\usepackage{amssymb}

\newcommand{\R}{\mathbb{R}}
\newcommand{\C}{\mathbb{C}}

\newcommand{\Z}{\mathbb{Z}}

\newcommand{\cc}{\mathbf{c}}

\newcommand{\End}{\text{End}}
\newcommand{\Hom}{\text{Hom}}
\newcommand{\Hol}{\text{Hol}}
\newcommand{\Pol}{\text{Pol}}

\newcommand{\Ker}{\text{Ker}}
\newcommand{\Impart}{\text{Im}}

\newcommand{\SL}{\mathbf{SL}}

\usepackage{cancel}
\theoremstyle{plain}

\numberwithin{equation}{section}

\theoremstyle{plain}
\newtheorem{cor}{Corollary}

\newtheorem{prop}{Proposition}
\newtheorem{thm}{Theorem}

\newtheorem{lem}{Lemma}
\newtheorem{defn}{Definition}
\theoremstyle{definition}
\newtheorem{exmp}{Example}

\theoremstyle{remark}
\newtheorem{req}{Remark}

\title{\textbf{Delorme's intertwining conditions for sections of homogeneous vector bundles on two and three dimensional hyperbolic spaces}}
\author{Martin \textsc{Olbrich} and Guendalina \textsc{Palmirotta}}
\date{} 

\begin{document}

\maketitle

\abstract
The description of the Paley-Wiener space for compactly supported smooth functions $C^\infty_c(G)$ on a semi-simple Lie group  $G$ involves certain intertwining conditions that are difficult to handle.
In the present paper, we make them completely explicit for $G=\mathbf{SL}(2,\mathbb{R})^d$ ($d\in \mathbb{N}$) and $G=\mathbf{SL}(2,\mathbb{C})$.
Our results are  based on a defining criterion for the Paley-Wiener space, valid for general groups of real rank one, that we derive from Delorme's proof of the Paley-Wiener theorem.
In a forthcoming paper, we will show how these results can be used to study solvability of invariant differential operators between sections of homogeneous vector bundles over  the corresponding symmetric spaces.
\tableofcontents

\section{Introduction}
Consider a Riemannian symmetric space of non-compact type $X=G/K$, where $G$ is a real connected semi-simple Lie group with finite center of non-compact type and $K \subset G$ is its maximal compact subgroup.\\
Arthur \cite{Arthur} as well as Delorme \cite[Thm. 2]{Delorme} established a Paley-Wiener theorem for ($K$-finite) compactly supported smooth functions on $G$. Their results involved the so-called Arthur-Campolli and Delorme conditions, respectively.
Later van den Ban and Souaifi \cite{vandenBan} proved, without using the proof nor validity of any associated Paley-Wiener theorem, that the two compatibility conditions are equivalent.\\
In \cite{PalmirottaPWS}, we proved a topological Paley-Wiener(-Schwartz) theorem for sections of homogeneous vector bundles by adapting Delorme's intertwining conditions for our purposes. We considered the intertwining conditions in three levels, namely (\textcolor{blue}{Level 1}) refered to Delorme's condition in the setting of van den Ban and Souaifi \cite{vandenBan}, (\textcolor{blue}{Level 2}) corresponded to the conditions for sections of homogeneous vector bundles and (\textcolor{blue}{Level 3}) stood for spherical functions.\\
However, these intertwining conditions are very difficult to check in practice, even for special $K$-types.\\
The most important source of such conditions are the Knapp-Stein \cite{KnappSteinI},\cite{KnappSteinII} and 
Zelobenko \cite{Zelo} intertwining operators, as well as the embedding of discrete series into principal series.

Therefore, in this article, we rewrite them in a more accessible way involving such intertwining operators and the Harish-Chandra $\mathbf{c}$-functions, which we introduce in the first Section~\ref{sect:Intwoperators}.
Moreover, in Section~\ref{sect:Meta}, we show that only a part of them is already sufficient for $G$ of real rank one (Thm~\ref{thm:Meta}).
This is already implicitly contained in Delorme's proof of the Paley-Wiener theorem \cite[Thm.~2]{Delorme}.
For our proof, we essentially use an intermediate result of Delorme \cite[Prop.~1]{Delorme} and his induction procedure \cite[ Prop.~2]{Delorme} on the length of minimal $K$-types of a generalized principal series representation.\\
To apply Thm.~\ref{thm:Meta}, one has to know more or less the complete composition series of reducible principle series representations, which is the case for the two special examples, $\mathbf{SL}(2,\mathbb{R})$ and $\mathbf{SL}(2,\mathbb{C})$, which are in the focus of the present paper.\\
In fact, for $G=\mathbf{SL}(2,\mathbb{R})$, in Section ~\ref{sect:SL2R}, by drawing its principal series representations $H^{\pm,\lambda}_\infty$ by '\textit{box-pictures}' (Fig.~\ref{fig:BoxSL2R}), we can see in which closed $G$-submodule of $H^{\pm,\lambda}_\infty$ there is an intertwining condition in (\textcolor{blue}{Level 2}) (Thm.~\ref{thm:Meta2SL2R}).
Afterwards, we can deduce the corresponding results also for the other levels (Thms.~\ref{thm:Meta3SL2R} \;\&~\ref{thm:Meta1}) and even go beyond by illustrating the intertwining conditions for finite products of $\mathbf{SL}(2,\mathbb{R})$ (Thm.~\ref{thm:Meta3SL4R}).
As a last example, in Section~\ref{sect:SL2C}, we consider $G=\mathbf{SL}(2,\mathbb{C})$.
The description of its intertwining conditions (Thms.~\ref{thm:Meta1SL2C},~\ref{thm:Meta2SL2C} \&~\ref{thm:Meta3SL2C}) is more difficult then for the previous examples. 
We remark that $\mathbf{SL}(2,\mathbb{R})$ and $\mathbf{SL}(2,\mathbb{C})$ are locally isomorphic to $\mathbf{SO}(1,n), n=2,3$ respectively.
In fact, based on Thm.~\ref{thm:Meta}, we have already checked that it is possible to obtain analogous results for all $n$.
We are grateful to P. Delorme for mentioning to us that the results for odd $n$ (including $\mathbf{SL}(2,\mathbb{C})$) can also be derived from his earlier Paley-Wiener theorem for groups with only one conjugacy class of Cartan subgroups \cite{DelormeSO2}.\\
This work is part of the second author's doctoral dissertation \cite{PalmirottaDiss}. In fact, our results can be used to study solvability of invariant differential operators between sections of homogeneous vector bundles on the corresponding symmetric spaces. 
The way we expressed the intertwining conditions in (\textcolor{blue}{Level 3}) is particularly adapted to the solvability questions.
In an upcoming paper \cite{PalmirottaSolv}, these questions will be discussed in detail.

\section{Intertwining conditions and operators}
\label{sect:Intwoperators}
Let $G$ be a real connected semi-simple Lie group with finite center of non-compact type with Lie algebra $\mathfrak{g}$. We fix an Iwasawa decomposition $G=KAN$, where
$K \subset G$ is a maximal compact subgroup with Lie algebra $\mathfrak{k}$, $A=\exp(\mathfrak{a})$ is abelian and $N$ is nilpotent.
The quotient $X=G/K$ is a Riemannian symmetric space of non-compact type.\\
Let $M=Z_K(\mathfrak{a})$ be the centralizer of $A$ in $K$.
Then $P=MAN$ is a minimal parabolic subgroup of $G$.
Let $(\sigma,E_\sigma) \in \widehat{M}$ be a finite-dimensional irreducible representation of $M$ and $\lambda \in \mathfrak{a}^*_\mathbb{C}$ the complexified dual of the Lie algebra of $A$.
For $(\sigma,\lambda) \in \widehat{M} \times \mathfrak{a}^*_\mathbb{C}$, consider 
$H^{\sigma,\lambda}_{\infty}$ the space of smooth vectors of the principal series representation of $G$ induced from the $P$-representation $\sigma_\lambda$ on the vector space $E_{\sigma}$ \cite[p. 168]{Knapp2}.
By $H^{\sigma}_\infty$, we denote its representation space in the compact picture, which is given by functions on $K$ and is independent of $\lambda$.
The corresponding operators $\pi_{\sigma,\lambda}(g) \in \text{End}(H^\sigma_\infty), g \in G$, depend holomorphically on $\lambda \in \mathfrak{a}^*_\mathbb{C}$.
In \cite{PalmirottaPWS}, we introduced the  $m$-th derived principal series representation of $G$. For similar slightly different definitions see also \cite[Def.~3 (4.4)]{Delorme} or \cite[Sect.~4.5]{vandenBan}.

\begin{defn}[$m$-th derived representation, \cite{PalmirottaPWS}, Def.~2] \label{def:succderv}
For $\lambda \in \mathfrak{a}^*_\mathbb{C}$, let $\emph \Hol_\lambda$ be the set of germs at $\lambda$ of $\mathbb{C}$-valued holomorphic functions $\mu \mapsto f_\mu$
and $m_\lambda \subset \emph \Hol_\lambda$ the maximal ideal of germs vanishing at $\lambda$.\\
Denote by $H^{\sigma}_{[\lambda]}$ the set of germs at $\lambda$ of $H^{\sigma}_\infty$-valued holomorphic functions 
$\mu \mapsto \phi_{\mu} \in H^{\sigma}_\infty$ 
with $G$-action
$$ 
(g\phi)_{\mu} = \pi_{\sigma,\mu}(g) \phi_{\mu},\;\;\; g \in G.
$$
For $m\in \mathbb{N}_0$, it induces a representation $\pi^{(m)}_{\sigma,\lambda}$ on the space
\begin{equation} \label{eq:succderv}
H^{\sigma,\lambda}_{\infty,(m)} := H^{\sigma}_{[\lambda]} / m_\lambda^{m+1} H^{\sigma}_{[\lambda]},
\end{equation}
which is equipped with the natural Fr\'echet topology.
We call this representation the $m$-th derived principal series representation of $G$.
\end{defn}

Let us first recall the intertwining conditions and corresponding Paley-Wiener theorems in the three levels from \cite{PalmirottaPWS}.\\

\noindent
\textbf{Intertwining conditions and Paley-Wiener theorem in (\textcolor{blue}{Level 1})}.
We denote by $\text{Hol}(\mathfrak{a}^*_\mathbb{C})$ the space of holomorphic functions on $\mathfrak{a}^*_\mathbb{C}$ and by $\text{Hol}(\mathfrak{a}^*_\mathbb{C}, \text{End}(H^\sigma_\infty))$ the space of maps
$\mathfrak{a}^*_\mathbb{C} \ni \lambda \mapsto \phi(\lambda) \in \text{End}(H^{\sigma}_\infty)$
such that for $\varphi\in H^{\sigma}_\infty,$ the function
$\lambda \mapsto \phi(\lambda)\varphi\in H^{\sigma}_\infty$
 is holomorphic.

\begin{defn}[Delorme's intertwining condition in (\textcolor{blue}{Level 1}), \cite{PalmirottaPWS}, Def.~3]
\label{defn:intwcond}
Let $\Xi$
be the set of all 3-tuples $(\sigma,\lambda,m)$ with $\sigma \in \widehat{M}$, $\lambda \in  \mathfrak{a}^*_\mathbb{C}$ and $m\in\mathbb{N}_0$.
Consider the $m$-th derived $G$-representation $H^{\sigma,\lambda}_{\infty,(m)}$ defined in (\ref{eq:succderv}). 
For every finite sequence $\xi=(\xi_1,\xi_2,\dots,\xi_s) \in \Xi^s, s\in \mathbb{N}$, we define the $G$-representation
$$H_\xi:= \bigoplus_{i=1}^s H^{\sigma_i,\lambda_i}_{\infty,(m_i)}.$$
We consider proper closed $G$-subrepresentations $W \subseteq H_\xi$.\\
Such a pair $(\xi,W)$ with $\xi \in \Xi^s$ and $W \subset H_\xi$ as above, is called an intertwining datum.
Every  function $\phi \in \prod_{\sigma\in \widehat{M}} \emph \Hol(\mathfrak{a}^*_\mathbb{C}, \emph \End(H^\sigma_\infty))$ induces an element
$$\phi_\xi \in \bigoplus_{i=1}^s \emph \End(H^{\sigma_i,\lambda_i}_{\infty,(m_i)}) \subset \emph \End(H_\xi).$$
\begin{itemize}
\item[(D.a)] We say that $\phi$ satisfies Delorme's intertwining condition, if $\phi_\xi(W) \subseteq W$ for every intertwining datum $(\xi,W)$.
\end{itemize}
\end{defn}

\noindent
Moreover, for $r> 0$, one can introduce a corresponding Paley-Wiener space
$$
PW_r(G) \subset \prod_{\sigma\in \widehat{M}} \text{Hol}(\mathfrak{a}^*_\mathbb{C}, \text{End}(H^\sigma_\infty)),
$$
which is characterised by the usual growth condition (e.g. \cite[Def.~3 (4.3)]{Delorme} or \cite[Def.~4 $(1.ii)_r$]{PalmirottaPWS}) and the intertwining condition $(D.a)$.
Let $\overline{B}_r(o)$ be the preimage of the closed ball of radius $r>0$ centered at $o\in X$ under the projection $p:G \rightarrow X$.
Write by $C^\infty_r(G)$ the
space of smooth complex functions on $G$ with support in $\overline{B}_r(o)$.
Then, by taking the union over all $r>0$, and considering the Fourier transform of $f\in~C^\infty_c(G)$:
\begin{eqnarray*}
\widehat{M} \times \mathfrak{a}^*_\mathbb{C} \in (\sigma,\lambda) \mapsto \mathcal{F}_{\sigma,\lambda}(f)&:=&\pi_{\sigma,\lambda}(f) \\
&=&\int_G f(g) \pi_{\sigma,\lambda}(g) \;dg \in \prod_{\sigma\in \widehat{M}} \text{Hol}(\mathfrak{a}^*_\mathbb{C}, \text{End}(H^\sigma_\infty))
\end{eqnarray*}

\begin{thm}[Delorme's Paley-Wiener Theorem in (\textcolor{blue}{Level 1}), \cite{Delorme}, Thm.~2] \label{thm:Delorme1}
For any $r>0$, the Fourier transform
$$ C_r^\infty(G) \ni f \mapsto \mathcal{F}_{\sigma,\lambda}(f)\in PW_r(G), \;\;\; (\sigma,\lambda) \in \widehat{M} \times \mathfrak{a}^*_\mathbb{C}$$
is a topological isomorphism between the Fr\'echet spaces $C_r^\infty(G)$ and $PW_r(G)$. \qed
\end{thm}

\noindent
\textbf{Intertwining conditions and Paley-Wiener theorems in (\textcolor{blue}{Level 2}) and (\textcolor{blue}{Level 3})}.
Next, consider two finite-dimensional, not necessary irreducible, $K$-representations $(\tau,E_\tau)$ and $(\gamma,E_\gamma)$.
We consider the space of smooth compactly supported sections of homogeneous vector bundles $\mathbb{E}_\tau$ over $X$ by
\begin{eqnarray*}
C_c^\infty(X,\mathbb{E}_\tau) 
&=& \bigcup_{r>0} C_r^\infty(X,\mathbb{E}_\tau) \\
&\cong& \bigcup_{r>0}
 \Big\{ f:G \stackrel{C^\infty}{\longrightarrow} E_\tau \;\vert\; f(gk)= \tau^{-1}(k)(f(g)),\\
&&\;\;\;\;\;\;\;\;\;\;\;\;\;\;\;\;\;\;\;\;\;\;\;\; \forall g \in, k \in K \text{ and } supp(f) \subset \overline{B}_r(0) \Big\}.
\end{eqnarray*}
The group $G$ acts on $C_c^\infty(X,\mathbb{E}_\tau)$ by left translations
$$(g \cdot f)(g')=f(g^{-1}g'), \;\;\;\;\;\;\; \forall g,g' \in G.$$
It is not difficult to see that we have the $G$-isomorphisms
$C^{\infty}(X, \mathbb{E}_\tau) \cong  C^\infty(G,E_\tau)^K \cong [C^{ \infty}(G) \otimes E_\tau]^K.$
Moreover, we also consider the space of $(\gamma,\tau)$-spherical functions on $G$
\begin{eqnarray*}
 C_c^\infty(G,\gamma,\tau) &=& \bigcup_{r>0} C_r^\infty(G,\gamma,\tau) \\
&:=& \bigcup_{r>0}  \Big\{ f : G \rightarrow \text{Hom}(E_\gamma,E_{\tau}) \;\vert\; f(k_1gk_2)=\tau(k_2)^{-1}f(g)\gamma(k_1)^{-1}, \\
&&\;\;\;\;\;\;\;\;\;\;\;\;\;\;\;\;\;\;\;\;\;\;\;\;\forall k_1,k_2 \in K \text{ and } supp(f) \subset \overline{B}_r(0) \Big\}.
\end{eqnarray*}
Note that by taking topological linear duals (of sections of dual bundles), we obtain the spaces of distributional sections $C^{-\infty}(X,\mathbb{E}_\tau)$ and $C^{-\infty}(G,\gamma,\tau)$, respectively.
We are particularly interested in distributional sections with compact support, i.e. in the spaces $C_c^{-\infty}(X,\mathbb{E}_\tau)$ and $C_c^{-\infty}(G,\gamma,\tau)$, respectively.

The Fourier transform for (distributional) sections of homogeneous vector bundles in (\textcolor{blue}{Level 2}) and (\textcolor{blue}{Level 3}) is given in the following definition.

\begin{defn}[Fourier transforms, \cite{PalmirottaPWS}, Def.~5 \& 6] \label{defn:FTsect}
Let $g=\kappa(g)a(g)n(g) \in KAN=G$ be the Iwasawa decomposition.
For fixed $\lambda \in \mathfrak{a}^*_\mathbb{C}$ and $k\in K$, we define the function $e^\tau_{\lambda,k}$ by
\begin{eqnarray} \label{eq:exptau}
e^\tau_{\lambda,k}: G &\rightarrow& \emph \End(E_\tau) \cong E_{\tilde{\tau}} \otimes E_\tau  \nonumber \\
g &\mapsto& e^\tau_{\lambda,k}(g):=\tau(\kappa(g^{-1}k))^{-1} a(g^{-1}k)^{-(\lambda+\rho)},
\end{eqnarray}
where $\rho$ is the half sum of the positive roots of $(\mathfrak{g},\mathfrak{a})$.
\begin{itemize}
\item[(a)] (\textcolor{blue}{Level 2}) For $f\in C^\infty_c(X,\mathbb{E}_\tau)$, the Fourier transformation is given by
\begin{equation} \label{eq:FTtau}
\mathcal{F}_\tau f(\lambda,k)=\int_G e^\tau_{\lambda,k}(g) f(g) \;dg = \int_{G/K} e^\tau_{\lambda,k}(g) f(g) \;dg,
\end{equation}
where the last equality makes sense, since the integrand is right $K$-invariant.
Similar, for distributional sections $T\in C^{-\infty}_c(X,\mathbb{E}_\tau)$:
$$\mathcal{F}_\tau T(\lambda,k):= \langle T, e^{\tau}_{\lambda, k} \rangle = T(e^\tau_{\lambda,k}) \in E_\tau,$$ with $(\lambda, k) \in \mathfrak{a}^*_\mathbb{C} \times K/M.$
\item[(b)] (\textcolor{blue}{Level 3}) The Fourier transformation for $f \in C^\infty_c(G,\gamma, \tau)$ is given by
\begin{equation} \label{eq:FTgammatau3}
\prescript{}{\gamma}{\mathcal{F}}_\tau f(\lambda):= \int_G e^\tau_{\lambda,1}(g) f(g) \;dg, \;\;\;\; \lambda \in \mathfrak{a}^*_\mathbb{C}.
\end{equation}
Similar, for $T \in C^{-\infty}_c(G,\gamma, \tau)$, we have
$ \prescript{}{\gamma}{\mathcal{F}}_\tau T(\lambda):= \langle T, e^\tau_{\lambda,1} \rangle.$
\end{itemize}
\end{defn}
Note that $\mathcal{F}_\tau(f)$, $\mathcal{F}_\tau(T) \in \text{Hol}(\mathfrak{a}^*_\mathbb{C}, H^{\tau\vert_M}_\infty)$, where
$$
H^{\tau\vert_M}_\infty :
= \{f:K \stackrel{C^\infty}{\rightarrow} E_\tau \;\vert\; f(km)=\tau(m)^{-1}f(k)\}
$$
and that $\prescript{}{\gamma}{\mathcal{F}}_\tau(f)$, $\prescript{}{\gamma}{\mathcal{F}}_\tau(T) \in \text{Hol}(\mathfrak{a}^*_\mathbb{C}, \text{Hom}_M(E_\gamma,E_\tau))$.
We also remark that $\text{Hom}_K(E_\gamma,H^{\tau\vert_M}_\infty) \cong \text{Hom}_M(E_\gamma,E_\tau),$ by Frobenius reciprocity.
We will also consider $H^{\tau\vert_M,\lambda}_{\infty,(m)}$ defined as in Def.~\ref{def:succderv} with $\sigma$ replaced by $\tau\vert_M$.

\begin{defn}
We define
$$
\emph \Hom_M(E_\tau,E_\sigma)^{\lambda}_{(m)}:= \emph \Hol(\mathfrak{a}^*_\mathbb{C}, \emph \Hom_M(E_\tau,E_\sigma))/m_\lambda^{m+1} \; \emph \Hol(\mathfrak{a}^*_\mathbb{C}, \emph \Hom_M(E_\tau,E_\sigma))
$$
as in (\ref{eq:succderv}).
For $\tau \in \widehat{K}$ and an intertwining datum $(\xi,W)$,
we set
\begin{eqnarray} \label{eq:DtauW}
D^\tau_W&:=& \Big\{t \in \bigoplus_{i=1}^s \emph \Hom_M(E_\tau,E_{\sigma_i})^{\lambda_i}_{(m_i)} \vert \nonumber \\
&\;\;\;\;\;\;\;\;\;\;\;\;\;\;\;&
\;\;\;\;\;\; T=\widetilde{Frob}^{-1}(t) \in \emph \Hom_K(E_\tau,W) \subset\emph \Hom_K(E_\tau, H_\xi) \Big\} \nonumber \\
&\subset&\bigoplus_{i=1}^s \emph \Hom_M(E_\tau,E_{\sigma_i})^{\lambda_i}_{(m_i)}.
\end{eqnarray}
\end{defn}

We have shown \cite[Def.~7,8; Prop.~7 \& Thm.~2]{PalmirottaPWS} that Delorme's intertwining condition $(D.a)$ corresponds to the following intertwining conditions in (\textcolor{blue}{Level~2}) and (\textcolor{blue}{Level 3}).

\begin{defn}[Intertwining conditions in (\textcolor{blue}{Level 2}) and (\textcolor{blue}{Level 3}), \cite{PalmirottaPWS}, Thm.~2 ] \label{defn:intcondL2L3}
Let $\Xi$ be as in Def.~\ref{defn:intwcond} and
 $\overline{\Xi}$ be the set of all tuples $(\lambda,m)$ with $\lambda \in \mathfrak{a}^*_\mathbb{C}$ and $m\in \mathbb{N}_0$.
We define a map
$$ \Xi \longrightarrow \overline{\Xi}, \;\;\; \xi=(\sigma,\lambda,m) \mapsto \overline{\xi}=(\lambda,m).$$
For $s\in  \mathbb{N}$ and $\xi \in \Xi^s$, we have the corresponding element $\overline{\xi}\in \overline{\Xi}^s$.
\begin{itemize}
\item[(D.2)] (\textcolor{blue}{Level 2})
We say that $\psi \in \emph \Hol(\mathfrak{a}^*_\mathbb{C}, H^{\tau\vert_M}_\infty)$ satisfies the intertwining condition, if for each intertwining datum $(\xi,W)$ and each non-zero $t=(t_1,t_2, \dots, t_s) \in D^\tau_W$, the induced element
$\psi_{\overline{\xi}} \in  \bigoplus_{i=1}^s H^{\tau\vert_M,\lambda_i}_{\infty,(m_i)}=:H^{\tau\vert_M}_{\overline{\xi}}$ satisfies
$$ t \circ \psi_{\overline{\xi}}=(t_1 \circ \psi_1, \dots, t_2 \circ \psi_s) \in W.$$
\end{itemize}
\begin{itemize}
\item[(D.3)] (\textcolor{blue}{Level 3})
We say that $\varphi \in  \emph \Hol(\mathfrak{a}^*_\mathbb{C},  \emph \Hom_M(E_\gamma,E_\tau))$  satisfies the intertwining condition, if for each intertwining datum $(\xi,W)$ and each non-zero $t=(t_1,t_2, \dots, t_s) \in D^\tau_W$, the induced element
$\varphi_{\overline{\xi}} \in  \bigoplus_{i=1}^s \emph \Hom_M(E_\gamma,E_\tau)^{\lambda_i}_{(m_i)} =:H^{\gamma,\tau}_{\overline{\xi}}$ satisfies
$$  t \circ \varphi_{\overline{\xi}}=(t_1 \circ \varphi_1, \dots, t_2 \circ \varphi_s) \in D^\gamma_W.$$
\end{itemize}
\end{defn}

\begin{exmp}[\cite{PalmirottaPWS}, Example~1] \label{exmp:IntwOp}
Consider now $s=2$ and $m_1=m_2=0$.
Let $$L: H^{\sigma_1,\lambda_1}_\infty \longrightarrow  H^{\sigma_2,\lambda_2}_\infty$$ 
be an intertwining operator between principal series representations.\\
The corresponding intertwining datum is given by
$\xi:=((\sigma_1,\lambda_1,0),(\sigma_2,\lambda_2,0)) \in \Xi^2$ and
$W=graph(L) \subset  H^{\sigma_1,\lambda_1}_\infty \oplus H^{\sigma_2,\lambda_2}_\infty$.
Moreover, define
$$l^\tau: \text{Hom}_M(E_\tau,E_{\sigma_1}) \longrightarrow \text{Hom}_M(E_\tau,E_{\sigma_2})$$ by
$$l^\tau(t)(v)=L(t \tau(\cdot)^{-1}v)(e)=L(\phi_v(t))(e),$$
where the element $\phi_v(t) \in H_\infty^\sigma$ is given by the function $\phi_v(t)(k):=t \tau(k^{-1})v,$
for $k \in K, v\in E_\tau$ and $t\in \text{Hom}_M(E_\tau,E_{\sigma_1})$.
Then
\begin{eqnarray*}
D^\tau_W=\{(t_1,t_2)\;\vert\; t_2=l^\tau(t_1)\} 
&=& \{(t,l^\tau(t))\;\vert\; t \in \text{Hom}_M(E_\tau,E_{\sigma_1})\} \\
&\subset& \text{Hom}_M(E_\tau,E_{\sigma_1}) \oplus \text{Hom}_M(E_\tau,E_{\sigma_2}).
\end{eqnarray*}
In this situation, we get the following intertwining conditions.
\begin{itemize}
\item[(D2.L)] (\textcolor{blue}{Level 2}) For each $t \in \text{Hom}_M(E_\tau,E_{\sigma_1})$, we have for $\psi(\lambda_i, \cdot) \in H^{\tau\vert_M, \cdot}_\infty, i=1,2$
\begin{equation} \label{eq:Lintwcond2}
 L( t\circ \psi(\lambda_1, \cdot))= l^\tau(t) \circ \psi(\lambda_2, \cdot).
\end{equation}
\item[(D3.L)] (\textcolor{blue}{Level 3})
For each $t \in \text{Hom}_M(E_\tau,E_{\sigma_1})$, we have for $\varphi(\lambda_i) \in \text{Hom}_M(E_\gamma,E_\tau), i=1,2$
\begin{equation} \label{eq:Lintwcond3}
 l^\gamma(t \circ \varphi(\lambda_1)) = l^\tau(t) \circ \varphi(\lambda_2).
\end{equation}
\end{itemize}
\end{exmp}

Similar as in (\textcolor{blue}{Level 1}), for $r>0$, we introduce a Paley-Wiener space in:
$$ \text{(\textcolor{blue}{Level 2}) } PW_{r,\tau}(\mathfrak{a}^*_\mathbb{C} \times K/M) \subset \text{Hol}(\mathfrak{a}^*_\mathbb{C}, H^{\tau\vert_M}_\infty)$$
and
$$ \text{(\textcolor{blue}{Level 3}) } \prescript{}{\gamma}{PW}_{\tau}(\mathfrak{a}^*_\mathbb{C}) \subset \text{Hol}(\mathfrak{a}^*_\mathbb{C}, \text{Hom}_M(E_\gamma,E_\tau)),$$
which is characterised by the usual growth condition \cite[Def.~$(2.ii)_r$ and $(3.ii)_r$, resp.]{PalmirottaPWS} and the corresponding intertwining condition $(D.2)$ and  $(D.3)$, respectively.
The Paley-Wiener-Schwartz space for distributional sections is denoted by $PWS_{\tau}(\mathfrak{a}^*_\mathbb{C} \times K/M)$ respectively $\prescript{}{\gamma}{PWS}_{\tau}(\mathfrak{a}^*_\mathbb{C})$.
It is characterised in the same way as above except that we replace the growth condition by a weaker one \cite[Def.~$(2.iis)_r$ and $(3.iis)_r$, resp.]{PalmirottaPWS}.
As in \cite[Sect.~6]{PalmirottaPWS} we equip the Paley-Wiener(-Schwartz) space with the inductive limit topology and the space of distributional sections $C_c^{-\infty}(X,\mathbb{E}_\tau)$ and $C^{-\infty}_c(G,\gamma, \tau)$ with the strong dual topology.
In \cite{PalmirottaPWS}, we derived the following.

\begin{thm}[Topological Paley-Wiener(-Schwartz) theorem for sections in (\textcolor{blue}{Level 2}) and (\textcolor{blue}{Level 3}), \cite{PalmirottaPWS} Thms.~3 \& 4] \label{thm:PWsect}
The Fourier transform
$$C_c^{\pm\infty}(X,\mathbb{E}_\tau) \ni \psi \mapsto \mathcal{F}_{\tau}(\psi)(\lambda, k)\in PW(S)_{\tau}(\mathfrak{a}^*_\mathbb{C} \times K/M), \;\;\; (\lambda,k) \in \mathfrak{a}^*_\mathbb{C} \times K$$
is a topological isomorphism between $C_c^{\pm\infty}(X,\mathbb{E}_\tau)$ and $PW(S)_{\tau}(\mathfrak{a}^*_\mathbb{C} \times K/M)$.\\
Moreover, by considering an additional $K$-representation $(\gamma,E_\gamma)$ with associated homogeneous vector bundle $\mathbb{E}_\gamma$, then the Fourier transform
$$ C^{\pm\infty}_c(G,\gamma, \tau) \ni \varphi \mapsto \prescript{}{\gamma}{\mathcal{F}}_{\tau}(\varphi)(\lambda) \in \prescript{}{\gamma}{PW(S)}_{\tau}(\mathfrak{a}^*_\mathbb{C}), \;\;\; \lambda \in \mathfrak{a}^*_\mathbb{C}$$
is a topological isomorphism between $C^{\pm \infty}_c(G,\gamma, \tau)$ and $\prescript{}{\gamma}{PW(S)}_{\tau}(\mathfrak{a}^*_\mathbb{C})$. \qed
\end{thm}

In Example~\ref{exmp:IntwOp}, we worked out the intertwining conditions coming from intertwining operators between principal series representations of $G$.
We want to make them more explicit for the most important case, the Knapp-Stein intertwining operators, which we now recall.\\

\noindent
\textbf{Knapp-Stein intertwining operator.}
Let $W_A:=N_K(\mathfrak{a})/M$ be the Weyl group.
Note that $W_A$ acts on $\mathfrak{a}^*_\mathbb{C}$ as well as on $\widehat{M}$.
Let $w\in W_A$ be represented by $m_w \in M':=N_K(\mathfrak{a})$ and $\sigma\in \widehat{M}$. 
We realise $\sigma$ on the vector space $E_\sigma.$
We define a new representation $w\sigma \in \widehat{M}$ of $M$ acting on $E_\sigma$ by
$$w\sigma(m):=\sigma(m_w^{-1}mm_w), \;\; m\in M.$$
Its equivalence class only depends on $w \in W_A$ and not on the choice of $m_w$.
\begin{defn}(Knapp-Stein intertwining operator, \cite{KnappSteinI}, \cite{KnappSteinII} \& \cite[ Chap. V]{Knapp1}) \label{def:KSintop}
Let $\Delta^+_\mathfrak{a}$ be the positive root system of $(\mathfrak{g},\mathfrak{a})$ corresponding to $N$.
Let $\Delta^-_\mathfrak{a}:=-\Delta^+_\mathfrak{a}$ and $\overline{N}$ be the unipotent subgroup coming from the associated Iwasawa decomposition corresponding to $\Delta^-_\mathfrak{a}$. Write $N_w:=N\cap w\overline{N}w^{-1}$.
For $(\sigma, \lambda) \in \widehat{M} \times \mathfrak{a}^*_\mathbb{C}$ with $(\text{Re}(\lambda), \alpha) >0$, for all $\alpha \in \Delta^+_\mathfrak{a} \cap w^{-1}\Delta^-_\mathfrak{a}$ 
and for a fixed representative $m_w \in M'$, we define the intertwining operator
\begin{equation*}\label{eq:KSintop}
J_{w,\sigma,\lambda}: H^{\sigma,\lambda}_\infty \longrightarrow H^{w\sigma,w\lambda}_\infty
\end{equation*}
by the convergent integral
$$
J_{w,\sigma,\lambda}(\varphi(g)):= \int_{{N}_w} \varphi(g{n}m_w)\;d{n}, \;\; g\in G, \varphi \in H^{\sigma,\lambda}_\infty,
$$
which depends holomorphically on $\lambda \in \mathfrak{a}^*_\mathbb{C}$.
This operator has a meromorphic continuation to $\mathfrak{a}^*_\mathbb{C}$.
\end{defn}

For later reference, let us state the intertwining conditions coming from the Knapp-Stein operators.

\begin{exmp}[Knapp-Stein intertwining condition in (\textcolor{blue}{Level 1}]
\label{exmp:Delormecond}
For $w\in W_A$ and fixed $(\sigma, \lambda) \in \widehat{M} \times \mathfrak{a}^*_\mathbb{C}$, we consider the Knapp-Stein intertwining operator $J_{w,\sigma,\lambda}$ as in Def.~\ref{def:KSintop}. Let
$\phi \in \prod_{\sigma \in \widehat{M}} \text{Hol}(\mathfrak{a}^*_\mathbb{C},\text{End}(H^\sigma_\infty))$. 
Then, the condition
\begin{equation} \label{eq:intwcondKS}
J_{w,\sigma,\lambda} \circ \phi(\sigma,\lambda) = \phi(w\sigma,w\lambda) \circ J_{w,\sigma,\lambda}
\end{equation}
is a special intertwining condition of $(D.a)$ in Def.~\ref{defn:intwcond}.
Note that the corresponding intertwining datum $(\xi,W)$ is given by
$\xi=((\sigma,\lambda,0),(w\sigma,w\lambda, 0))$ and
$W=graph(J_{w,\sigma,\lambda}) \subset H^{\sigma,\lambda}_{\infty} \oplus H^{w\sigma,w\lambda}_{\infty},$ (compare Example~\ref{exmp:IntwOp}).
\end{exmp}

\begin{defn}[Harish-Chandra $\cc$-function, (e.g. \cite{Martin}, Def. 3.8)] \label{def:cfct}
Let $\tau \in \widehat{K}, w\in W_A$, $\overline{N}_w:=\overline{N} \cap w^{-1}Nw$ and $\lambda \in \mathfrak{a}^*_\mathbb{C}$ with
$(\text{Re}(\lambda), \alpha ) >0$, for all $\alpha \in \Delta^+_\mathfrak{a} \cap w^{-1}\Delta_\mathfrak{a}^-$. The $\mathbf{c}$-function is defined by
$$\mathbf{c}_{w,\tau}(\lambda):=\int_{\overline{N}_w} a(\overline{n})^{-(\lambda + \rho)}\tau(\kappa(\overline{n})) \; d\overline{n} \in \text{End}_M(E_\tau),$$
which can be extended to a meromorphic function on $\mathfrak{a}^*_\mathbb{C}$.\\
Furthermore, we have 
$\mathbf{c}_{w,\tau}(\sigma,\lambda):=pr_\sigma \circ \mathbf{c}_{w,\tau}(\lambda) \circ pr_\sigma \in \text{End}_M(E_\tau(\sigma)),$
where $pr_\sigma:E_\tau \longrightarrow E_\tau(\sigma)$ is the projection on the $\sigma$-isotypic compenent.\\
Consider now $w\in W_A$ as a Weyl element with maximal length, then, we set
$$\mathbf{c}_\tau(\lambda):=\mathbf{c}_{w,\tau}(\lambda) \;\;\;\;\; \text{ and }\;\;\;\;\; \mathbf{c}_\tau(\sigma,\lambda):=\mathbf{c}_{w,\tau}(\sigma,\lambda).$$
\end{defn}

\noindent
One can express the Knapp-Stein intertwining operators, by the Harish-Chandra $\mathbf{c}$-functions \cite[Lem. 3.12 \& Satz 3.13]{Martin}:
\begin{eqnarray} \label{eq:normKSintw}
J_{w,\sigma,\lambda}(\phi_v(t))
=\phi_v(t \circ \tau(m_w^{-1})\mathbf{c}_{w^{-1},\tau}(-w\lambda)),
\end{eqnarray}
where $\phi_v(t)$ is defined as in Example~\ref{exmp:IntwOp} for $v \in E_\tau$ and $t\in \text{Hom}_M(E_\tau,E_\sigma).$
Set $J_{w,\tau,\lambda}:= \tau(m_w)J_{w,\tau\vert_M,\lambda}$, which is independent of the choice of $w\in W_A$.
Hence, this leads to the following statement.

\begin{prop}[Knapp-Stein intertwining condition in (\textcolor{blue}{Level 2}) and (\textcolor{blue}{Level 3})] \label{prop:normKSLevel2}
\noindent
\begin{itemize}
\item[(a)] (\textcolor{blue}{Level 2}) Let $\psi \in \emph \Hol(\mathfrak{a}^*_\mathbb{C}, H_\infty^{\tau\vert_M})$ satisfying $(D.2)$.
Then, we have
\begin{equation} \label{eq:intwcondLevel2}
J_{w,\tau,\lambda} \psi(\lambda, \cdot) = \mathbf{c}_{w^{-1},\tau}(-w\lambda) \psi(w\lambda, \cdot), \;\;\; \lambda \in \mathfrak{a}^*_\mathbb{C}, w \in W_A.
\end{equation}
\item[(b)] (\textcolor{blue}{Level 3}) Let $\varphi \in \emph \Hol(\mathfrak{a}^*_\mathbb{C}, \emph \Hom_M(E_\gamma,E_\tau))$ satisfying $(D.3)$.
Then, we have
\begin{equation} \label{eq:normKSintwop3}
\varphi(\lambda) \gamma(m_w^{-1}) \mathbf{c}_{w^{-1},\gamma}(- w\lambda)=\tau(m_w^{-1})\mathbf{c}_{w^{-1},\tau}(-w\lambda) \circ \varphi(w\lambda),\;\;\; \lambda \in \mathfrak{a}^*_\mathbb{C}, w \in W_A.
\end{equation} \qedhere
\end{itemize}
\end{prop}

\begin{proof}
We consider the operator $l^\tau: \text{Hom}_M(E_\tau,E_\sigma) \longrightarrow \text{Hom}_M(E_\tau,E_\sigma)$ as in Example~\ref{exmp:IntwOp} associated by $L=J_{w,\sigma,\lambda}$.
Then (\ref{eq:normKSintw}) says that
$$l^\tau (t)=t \circ \tau(m_w^{-1}) \circ c_{w^{-1},\tau}(-w\lambda).$$
Now the proposition follows from (\ref{eq:Lintwcond2}) and (\ref{eq:Lintwcond3}), as in Example~\ref{exmp:IntwOp}.
\end{proof}

\begin{exmp} 
\begin{itemize}
\item[(a)] Let $\tau$ be a trivial one-dimensional represenation. Then,
 $C_c^\infty(X,\mathbb{E}_\tau)=C^\infty_c(X)$ and $H^{\tau\vert_M}_\infty=C^\infty(K/M)$.
Helgason showed in \cite[Thm. 5.1.]{Helgason3} that $\beta \in \text{Hol}(\mathfrak{a}^*_\mathbb{C}, C^\infty(K/M))$ belongs to the Paley-Wiener space if, and only, if it satisfies the usual growth condition and 
the intertwining condition:
\begin{equation} \label{eq:Helgasonintw}
\int_{K/M} e_{w\lambda,k}^\tau(g) \beta(w\lambda,k) dk = \int_{K/M} e_{-\lambda,k}^\tau(g) \beta(\lambda,k) dk, \;\;\;\;\; w\in W_A.
\end{equation}
It is not difficult to show that Helgason's intertwining condition is equivalent to (\ref{eq:intwcondLevel2}).
In fact, for $\lambda \in \mathfrak{a}^*_\mathbb{C}$, consider the Poisson transform (e.g. \cite{Helgason3} or \cite[Def. 3.2]{Martin})
$\mathcal{P}_{\lambda}: C^\infty(K/M) \longrightarrow C^\infty(X)$
given by
$
\mathcal{P}_{\lambda}(f)(g):= \int_K e^{\tau}_{-\lambda,k}(g) f(k) \; dk,$ for $g \in G$.
Then, Helgason's condition (\ref{eq:Helgasonintw}) can be expressed in terms of Poisson transform
$$
\mathcal{P}_{\lambda} \circ \beta_\lambda = \mathcal{P}_{w\lambda} \circ \beta_{w\lambda}, \;\;\;\;\; w\in W_A, \lambda \in \mathfrak{a}^*_\mathbb{C},
$$
where $\beta_\lambda:=\beta(\lambda, \cdot)$.
The result now follows from the functional equation of the Poisson transform \cite[Satz~3.15]{Martin}.
\item[(b)]
Let $\tau$ and $\gamma$ be two trivial one-dimensional representations.
Consider a function $\beta\in \text{Hol}(\mathfrak{a}^*_\mathbb{C})$ which satisfies the usual growth condition.
Helgason \cite[Thm. 7.1]{Helgason2} and Gangolli \cite{Gangolli} proved that $\beta \in \prescript{}{\gamma}{PW}_\tau(\mathfrak{a}^*_\mathbb{C})$, if and only, if
$$
\beta(\lambda)=\beta(w\lambda), \text{ for } \lambda\in \mathfrak{a}^*_\mathbb{C}, w\in W_A.
$$
This condition is equivalent to (\ref{eq:normKSintwop3}) in the case $\gamma, \tau$ are trivial.
\end{itemize}
\end{exmp}

\section{Sufficient intertwining conditions for rank one} \label{sect:Meta}
We want to reduce the amount of intertwining data in $(D.a)$ of Def.~\ref{defn:intwcond} to a minimum.
In this section, we assume that  $G$ has real rank one.
In this situation, the set of positive restricted roots $\Delta^+_\mathfrak{a}$ consists of at most two elements, namely $\alpha$ and possibly $2\alpha$. 
The Weyl group is reduced to $\{-1,1\}$ acting on $\mathfrak{a}^*_\mathbb{C}$ by multiplication.

We have the following special intertwining conditions.
An irreducible unitary representation $(\pi,E_\pi)$ of $G$ is called a representation of the discrete series if there is a $G$-invariant embedding $E_\pi \hookrightarrow L^2(G)$. 
Here, $L^2(G)$ denote the space of all square integrable functions with respect to invariant measure $dg$ on $G$.
Write $\widehat{G}_d$ for the set of equivalence classes of discrete series representations of $G$.
Let $H_{\pi}$ be any Hilbert space, where the representation $\pi \in \widehat{G}_d$ is realized.
Let $H^\infty_\pi \subset H_\pi$ be the corresponding space of smooth vectors.
For every representation of the discrete series $\pi \in \widehat{G}_{d}$, we choose an embedding
$$i_\pi: H^\infty_{\pi} \hookrightarrow H^{\sigma_\pi,\lambda_\pi}_\infty$$
into some principal series representation (Casselman's subrepresentation theorem \cite[Thm. 3.8.3.]{Wallach} and Casselman's and Wallach's globalization theorem \cite[Chap. 11]{Wallach2}) and set
$$W_\pi:= i_\pi(H^\infty_{\pi}) \subset H^{\sigma_\pi,\lambda_\pi}_\infty.$$
It is a closed $G$-invariant subspace.
Hence, the condition
\begin{equation} \label{eq:intwconddiscrete}
\phi(\sigma_\pi,\lambda_\pi)( W_\pi) \subset  W_\pi,\;\;\; \pi \in  \widehat{G}_{d}
\end{equation}
 is also of the form $(D.a)$, with $s=1$ and $m=0$, and it permits us to define
\begin{equation} \label{eq:phipi}
\phi(\pi):=\phi(\sigma_\pi,\lambda_\pi)\vert_{W_\pi} \in \text{End}(W_\pi).
\end{equation}

For $r>0$, we define the 'special' Paley-Wiener space $PW^+_r(G)$ by replacing Delorme's intertwining condition $(D.a)$ by the conditions (\ref{eq:intwcondKS}) and (\ref{eq:intwconddiscrete}), only.\\
Let $w\in W_A$ be the non-trivial element. For $\lambda \in \mathfrak{a}^*_\mathbb{C}$ with $(\text{Re}(\lambda), \alpha ) >0$,
let $m\in \mathbb{N}_0$ be the maximal order of the zeros of 
$J_{w,\sigma,\mu}(f_\mu)$ at $\mu=\lambda$, where $\mu \mapsto f_\mu \in H^{\sigma,\mu}_\infty$ runs over all germs of holomorphic functions at $\lambda$ with $f_\lambda \neq 0$.\\
We consider the induced operator
\begin{equation} \label{eq:m}
J_{w,\sigma,\lambda}^{(m-1)}: H^{\sigma,\lambda}_{\infty,(m-1)} \longrightarrow H^{w\sigma,-\lambda}_{\infty,(m-1)}
\end{equation}
and its kernel
$\text{Ker}(J_{w,\sigma,\lambda}^{(m-1)}) \subset H^{\sigma,\lambda}_{\infty,(m-1)}.$
By convention, we set $H^{\sigma,\lambda}_{\infty,(-1)}:=\{0\}$, for $m=0$.
Notice that due to condition (\ref{eq:intwcondKS}), we have for $\phi^{(m-1)}(\sigma,\lambda) \in \text{End}(H^{\sigma,\lambda}_{\infty,(m-1)})$:
$$\phi^{(m-1)}(\sigma,\lambda)(\text{Ker}(J_{w,\sigma,\lambda}^{(m-1)})) \subset \text{Ker}(J_{w,\sigma,\lambda}^{(m-1)}) \subset  H^{\sigma,\lambda}_{\infty,(m-1)}, \;\;\;\; (\sigma,\lambda) \in \widehat{M} \times \mathfrak{a}^*_\mathbb{C}.$$

Let $\tau \in \widehat{K}$ with highest weight $\mu_\tau \in i\mathfrak{t}^*$, where $\mathfrak{t} \subset \mathfrak{k}$ is the Lie algebra of a maximal torus $T \subset K$. We define
$$2\rho_c := \sum_{\alpha \in \Delta^+(\mathfrak{k},\mathfrak{t})} \alpha \in i\mathfrak{t}^*,$$
being the sum of all positive roots of $\mathfrak{t}_\mathbb{C}$ in $\mathfrak{k}_\mathbb{C}$.
For a $K$-representation $V$, we denote by $V(\tau)$ its corresponding isotypic component.
For $\sigma \in \widehat{M}$ and $\pi \in \widehat{G}_d$, we define $\vert\sigma\vert, \vert\pi\vert \in [0,\infty]$ by
$$\vert\sigma\vert:=\min_{\{\tau \;\vert\; H^\sigma_\infty(\tau) \neq \{0\}\}} \vert\vert\mu_\tau +2\rho_c\vert\vert  \text{ and } \vert\pi\vert:=\min_{\{\tau \;\vert\;  H^\infty_\pi(\tau) \neq \{0\}\}}\vert\vert\mu_\tau +2\rho_c\vert\vert,$$
i.e. $\vert\sigma\vert, \vert\lambda\vert$ are the lengths
of 'the' minimal $K$-type $\tau$ of $H^\sigma_\infty$ and $H^\infty_{\pi}$, respectively \cite[Sect. 1.3]{Delorme}.
Denote by $B(\sigma), B(\pi) \subset \widehat{K}$ the finite set of all minimal $K$-types of $H^\sigma_\infty, H^\infty_\pi.$

\begin{exmp} \label{exmp:SL2RM}
Let $G=\mathbf{SL}(2,\mathbb{R})$ and $K=\mathbf{SO}(2)$ its maximal compact subgroup. With the notations introduced in Sect.~\ref{sect:SL2R} below, we have that $\widehat{K} \cong \mathbb{Z}$ and $\rho_c=0.$
\begin{itemize}
\item[(i)] $M=\{\pm 1\}$, thus
\begin{itemize}
\item[-] if $\sigma$ is trivial, then $B(\sigma)=\{0\} \subset \mathbb{Z}$ (trivial $K$-type), and $\vert\sigma\vert=0,$
\item[-] if $\sigma$ is non-trivial, then $B(\sigma)=\{+1,-1\} \subset \mathbb{Z}$ and $\vert\sigma\vert=1$.
\end{itemize}
\item[(ii)] Let $\pi =D_{k}, k\in \mathbb{Z}\backslash \{0\}$ be the discrete series representation of $G$ parametrized as in Thm.~\ref{thm:Bargmann}  below, then
$$
B(D_{k})=\begin{cases} \{ k+1 \}, & k>0, \\  \{ k-1 \}, & k<0 \end{cases}$$
and $\vert\pi\vert=\vert k \vert +1$.
\end{itemize}
\end{exmp}

The following result tells us that the intertwining conditions (\ref{eq:intwcondKS}) and (\ref{eq:intwconddiscrete}) with an additional 'vanishing' condition are sufficient for semi-simple Lie groups $G$ of real rank one.

\begin{thm} \label{thm:Meta}
Let $\text{rk}_\mathbb{R}(G)=1$. For $r>0$, let $\mathcal{A}$ be a linear closed and $K\times K$-invariant subspace of $PW^+_r(G)$ satisfying $\mathcal{F}_{\sigma,\lambda}(C^\infty_r(G)) \subset \mathcal{A}$ and the following condition:
\begin{itemize}
\item[(D.b)] Let $\sigma \in \widehat{M}$ and $\phi \in \mathcal{A}$ such that
\begin{itemize}
\item[(i)] $\phi(\sigma', \lambda)=0$, for all $\sigma'\in \widehat{M}$ with $\vert\sigma'\vert >\vert\sigma\vert$ and $\lambda \in \mathfrak{a}^*_\mathbb{C}$,
\item[(ii)] $\phi(\pi)=0$, for all $\pi \in \widehat{G}_d$ with $\vert\pi\vert>\vert\sigma\vert$.
\end{itemize}
Then, for all $\lambda \in \mathfrak{a}^*_\mathbb{C}$ with $(\text{Re}(\lambda), \alpha ) >0$, $\phi$ induces the zero-operator on $ \emph \Ker(J_{w,\sigma,\lambda}^{(m-1)})$:
$$\phi^{(m-1)}(\sigma,\lambda)\Big\vert_{\emph \Ker(J_{w,\sigma,\lambda}^{(m-1)})}=0.$$ Here, $(m-1)$ depends on $(\sigma,\lambda)$ as defined above (\ref{eq:m}) and $\phi(\pi)$ is defined in (\ref{eq:phipi}).
\end{itemize}
Then, $$\mathcal{A}=PW_r(G) \cong \mathcal{F}_{\sigma,\lambda}(C^\infty_r(G)).$$
\end{thm}

\begin{proof}[Proof of Thm.~\ref{thm:Meta}]
By Delorme's Paley-Wiener Thm.~\ref{thm:Delorme1}, we already know that $PW_r(G)\cong\mathcal{F}_{\sigma,\lambda}(C^\infty_r(G))$ is a closed and $K\times K$-invariant subspace of $PW^+_r(G)$.
Therefore, it suffices to show that $\mathcal{F}_{\sigma,\lambda}(C^\infty_r(G))\subset \mathcal{A}$ is dense.
Thus for every $K\times K$-finite element $\phi \in \mathcal{A}$, we need to find a function $f\in C^\infty_r(G)_{K\times K}$ such that
$$\pi_{\sigma,\lambda}(f)=\phi(\sigma,\lambda), \;\;\;\; \forall (\sigma,\lambda) \in \widehat{M} \times \mathfrak{a}^*_\mathbb{C}.$$
Let $\phi \in \mathcal{A}_{K\times K}$. It is given by a collection $(\phi_\sigma), \sigma \in \widehat{M}$.
By $K\times K$-finiteness, only finitely many $\phi_\sigma$ are non-zero.
Similar, by $K\times K$-finiteness, $\phi(\pi)=0$, for all but finitely many $\pi \in \widehat{G}_d$. Indeed, for any given $K$-type $\tau$, there are only finitely many $\pi \in \widehat{G}_d$, with $H^\infty_\pi(\tau) \neq 0$ (e.g. \cite[Cor. 7.7.3]{Wallach}).\\
We define $l(\phi) \in [0,\infty)$ by
$$l(\phi):=\max \{\vert\sigma\vert, \vert\pi\vert \;\vert\; \sigma \in \widehat{M}, \phi_\sigma \neq 0; \pi \in \widehat{G}_d, \phi(\pi) \neq 0\}.$$
We can now imitate the inductive proof of Prop. 2 in Delorme's paper \cite{Delorme}.\\
Assume, as induction hypothesis, that for all $\psi \in \mathcal{A}$ with $l(\psi) < l(\phi)$, there are $f\in C^\infty_r(G)$ with $\mathcal{F}(f)=\psi.$
We enumerate
$$ \{\sigma \in \widehat{M} \;\vert\; \vert\sigma\vert=(\phi) \}=\{\sigma_1, \dots, \sigma_n\} \cup \{w\sigma_1, \dots, w\sigma_n\}$$
and
$$ \{\pi \in \widehat{G}_d \;\vert\; \vert\pi\vert=(\phi) \}=\{\pi_1, \dots, \pi_s\}.$$
Condition $(ii)$ together with (\ref{eq:intwcondKS}) says, in particular, that $\phi_{\sigma_i}$ belongs to a space that Delorme denotes by $\mathcal{K}_{\sigma_i}$, \cite[Def.~1]{Delorme}.
Strictly speaking Delorme has a condition for $(\text{Re}(\lambda), \alpha ) \geq 0.$
But if $\text{rk}_\mathbb{R}(G)=1$, only $(\text{Re}(\lambda), \alpha ) >0$ matters.
Note, that $\phi(\pi_j)$ belongs automatically to $\mathcal{K}_{\pi_j}$.
We can apply Prop.~1 together with Eq. (1.38) of Delorme's paper \cite{Delorme}, to deduce the existence of
$f_1,f_2,\dots,f_n\in C^\infty_r(G)$ and  $g_1,g_2,\dots,g_s \in C^\infty_r(G)$ with
\begin{eqnarray*}
\pi_{\sigma_i,\lambda}(f_i)&=& \phi(\sigma_i,\lambda), \;\;\; i\in\{1,2,\dots,n\},\\
\pi_j(g_j)&=& \phi(\pi_j), \;\;\;\;\;\; j\in\{1,2,\dots,s\},
\end{eqnarray*}
for $\lambda \in \mathfrak{a}^*_\mathbb{C}$.
Moreover, the discussion after Eq. (3.9) in \cite[p.1018]{Delorme}, makes clear that we can choose the $f_i$ and $g_j$ such that
\begin{itemize}
\item[(i)] $l(\mathcal{F}(f_i))=l(\mathcal{F}(g_j))=l(\phi), \forall i\in\{1,2,\dots,n\},  j\in\{1,2,\dots,s\}$ and
\item[(ii)] $\pi_{\sigma_k,\lambda}(f_i)=0, \forall k \neq i,$
\item[(iii)] $\pi_{\sigma_i,\lambda}(g_j)=0, \forall i,j,$
\item[(iv)] $\pi_{j}(f_i)=0, \forall i,j,$
\item[(v)] $\pi_{k}(g_j)=0, \forall k \neq j.$
\end{itemize}
Now, we set
$$\psi:=\phi-\sum_{i=1}^n \mathcal{F}(f_i)-\sum_{j=1}^s \mathcal{F}(g_j).$$
Then, by (i)-(v) we have $l(\psi) < l(\phi)$. Thus, by induction hypothesis $\psi=\mathcal{F}(f_0)$.
We conclude that $\phi=\mathcal{F}(f)$ with
$f=f_0+f_1+\cdots+f_n+g_1+g_2+\cdots+g_s.$
\end{proof}

\begin{req}
The result above can be extended to higher real rank. 
The extension involves representations induced from all cuspidal parabolic subgroups $P$ as well as the Knapp-Stein intertwining operator for them.
\end{req}

\section{The case $G=\SL(2,\R)$ and beyond} \label{sect:SL2R}
We consider
$G=\mathbf{SL}(2,\mathbb{R})=\Big\{g:=\begin{pmatrix} a & b \\ c& d \end{pmatrix} \in \mathbf{GL}(2,\mathbb{R}) \;\Big\vert\; \text{det}(g)=1 \Big\},$
 the special linear group of $\mathbb{R}^2$. It has dimension three.
We fix the Iwasawa decomposition $G=KAN$, where
$$
K=\mathbf{SO}(2)=\Bigg\{k_\theta:=\begin{pmatrix} \cos \theta & \sin \theta \\ -\sin \theta & \cos \theta \end{pmatrix} \;\Big\vert\; \theta \in \mathbb{R}\Bigg\}, \;\;\; A=\Bigg\{a_t:=\begin{pmatrix} e^t & 0 \\ 0 & e^{-t} \end{pmatrix} \;\Big\vert\; t \in \mathbb{R} \Bigg\},
$$
$$N=\Bigg\{n_x:=\begin{pmatrix} 1 & x\\ 0 & 1\end{pmatrix} \;\Big\vert\; x \in  \mathbb{R} \Bigg\}.$$
$G$ is a connected and simple Lie group with maximal compact subgroup $K$.
Clearly, $K$ is isomorphic to the unit circle $\mathbb{S}^1$.
Hence
$$\widehat{K}=\{\delta_n\;\vert\; n\in  \mathbb{Z}\} \cong \mathbb{Z}, \;\; \delta_n(k_\theta):=e^{in\theta} \in \mathbf{GL}(1,\mathbb{C})\cong\mathbb{C}\backslash \{0\}.$$ 
The representation space $E_{\delta_n}$ is one-dimensional and equal to $\mathbb{C}$.
We sometimes denote the $K$-representation $\delta_n$ simply by $n$.
If $H=\begin{pmatrix} t & 0 \\ 0 & -t \end{pmatrix} \in \mathfrak{a}$, then the positive root $\alpha$ is given by $\alpha(H)=2t$ and $\rho(H)=t$. 
We identify $\mathfrak{a}^*_\mathbb{C}$ with $\mathbb{C}$ via $z\alpha \mapsto z$.
In particular, $\rho \mapsto \frac{1}{2}$.

Since $M=\{\pm \text{Id}\}$, we have $\widehat{M} \cong \mathbb{Z}/2\mathbb{Z}$.
Let $\{+\}$ be the trivial and $\{-\}$ the non-trivial element of $\widehat{M}$.
For $\sigma=\{\pm\} \in \widehat{M}$ and $\lambda \in \mathbb{C} \cong \mathfrak{a}^*_\mathbb{C}$, we write $(\pi_{\pm,\lambda}, H^{\pm,\lambda}_\infty)$ for the principal series representations of $G$.
Its restriction to $K$ is the set of Fourier series on $\mathbb{S}^1$ with only non-zero even or odd Fourier coefficients
$$
H^{\pm}_\infty = \{f \in C^\infty(K,\mathbb{C}) \;\vert\; f \text{ even or odd }\}
\stackrel{K}{\cong} \bigoplus_{n \text{ even or odd}} \delta_n.$$
In order to write down the composition series of $H^{\pm,\lambda}_\infty$, we
 will denote an exact, non-splitting module sequence
$$ 0 \longrightarrow A  \longrightarrow B  \longrightarrow C  \longrightarrow 0$$
shortly by a '\textit{boxes-picture}'
\begin{center}
\begin{tikzpicture}[scale=0.5]
\draw (-1,1.5) -- (-1,1.5) node[anchor=north] {$B=$};
\draw (0,0) rectangle (1.5,1) node[pos=.5] {$A$};
\draw (0,1) rectangle (1.5,2)node[pos=.5] {$C$};
\draw (2,1.3) -- (2,1.3) node[anchor=north] {$.$};
\end{tikzpicture}
\end{center}
A proof of the following classical result can be found for example in \cite[ 5.6]{Wallach} or in \cite[Ch. VI]{Lang}. 
Note that the referenced proofs are also valid for $G$-representations of smooth vectors instead of $(\mathfrak{g},K)$-modules, if we apply Casselman's and Wallach's globalization theorem \cite[Prop.~11]{Wallach2}.

\begin{thm}[Structure of principal series representations of $\SL(2,\R)$]  \label{thm:Bargmann} 
The principal series representation $H^{\pm,\lambda}_\infty$ of $\mathbf{SL}(2,\mathbb{R})$ is reducible if and only if
$$
\lambda \in I_\pm := \begin{cases}
\frac{1}{2} +\mathbb{Z}, & \sigma=+,\\
\Z, & \sigma=-.
\end{cases}$$
For $\lambda=\frac{k}{2}\in I_{\pm}, k\in \mathbb{N}$, we have
\begin{center}
\begin{tikzpicture}[scale=0.5]
\draw (-1.7,1.5) -- (-1.7,1.5) node[anchor=north] {$H^{\pm,-\frac{k}{2}}_\infty=$};
\draw (0,0) rectangle (4,1) node[pos=.5] {$F_k$};
\draw (0,1) rectangle (4,2) node[pos=.5] {$D_{-k} \oplus D_{k}$};

\draw (7.5,1.5) -- (7.5,1.5) node[anchor=north] {$H^{\pm,\frac{k}{2}}_\infty=$};

\draw (9,0) rectangle (13,1) node[pos=.5] {$D_{-k} \oplus D_{k}$};
\draw (9,1) rectangle (13,2) node[pos=.5] {$F_k$};
\end{tikzpicture}
\end{center}
where $F_k:=\bigoplus^{k-1}_{l=-(k-1)} \delta_{2l}$ is the finite-dimensional irreducible $\mathbf{SL}(2,\mathbb{R})$-representation of dimension $k$ and $D_{\pm k}$ is the space of vectors of a representation of discrete series, which is characterized by the $K$-type decomposition
$D_{\pm k}= \overline{\bigoplus}_{j\geq0} \delta_{\pm(k+1+2j)}.$
Furthermore, for $\lambda=0$, we have
\begin{center}
\begin{tikzpicture}[scale=0.5]
\draw (-1.7,1.2) -- (-1.7,1.2) node[anchor=north] {$H^{-,0}_\infty=$};
\draw (0,0) rectangle (4,1) node[pos=.5] {$D_- \oplus D_{+}$};
\end{tikzpicture}
\end{center}
where
$D_{\pm}=\overline{\bigoplus}_{j\geq0} \delta_{\pm(1+2j)}$ are the limits of the discrete series. \qed\end{thm}

\begin{req} \label{req:Wsubmodule}
Let $W_\lambda$ be a proper closed invariant $G$-submodule of $H^{\pm,\lambda}_\infty$ for $\lambda \in I_\pm$. Then, Thm.~\ref{thm:Bargmann} tells us that
Then, one can observe that
\begin{itemize}
\item[-] for $\lambda >0$, $W_\lambda \in \{D_{-k}, D_k, D_{-k} \oplus D_{k}\}$, $k=2\lambda$,
\item[-] for $\lambda <0$, $W_\lambda \in \Big\{F_k, 
\begin{tikzpicture}[scale=0.5]
\draw (5,0) rectangle (7,1) node[pos=.5] {$F_k$};
\draw (5,1) rectangle (7,2) node[pos=.5] {$D_{-k}$};
\end{tikzpicture}
\;\;,\;
\begin{tikzpicture}[scale=0.5]
\draw (5,0) rectangle (7,1) node[pos=.5] {$F_k$};
\draw (5,1) rectangle (7,2) node[pos=.5] {$D_{k}$};
\end{tikzpicture}
\Big\}$, $k=2\lambda$,
\item[-] while for $\lambda=0$ and $\sigma=-$, $W_\lambda\in \{D_+,D_-\}.$
\end{itemize}
\end{req}

\noindent
To describe the intertwining conditions for $G=\mathbf{SL}(2,\mathbb{R})$ in the three levels,
we first need some preparation. The Harish-Chandra $\mathbf{c}$-function for $G$ is denoted by $\mathbf{c}_n(\lambda)$ for $n \in \mathbb{Z}$.
Due Cohn \cite[App. 1]{Cohn}, it is given explicitly in terms of gamma function $\Gamma(\cdot)$, by the formula (for a suitable normalization of the Haar measure $d\overline{n}$):
\begin{equation} \label{eq:cnSL(2,R)}
\mathbf{c}_n(\lambda) = \mathbf{c}_{-n}(\lambda) = \frac{1}{\sqrt{\pi}} \frac{\Gamma(\lambda)\Gamma(\lambda+\frac{1}{2})}{\Gamma(\lambda+\frac{1+n}{2})\Gamma(\lambda+\frac{1-n}{2})}, \;\;\; \lambda \in \mathfrak{a}_\mathbb{C}^*.
\end{equation}
Let $n \equiv m$ (mod 2). Using the gamma function recurrence formula 

\begin{equation} \label{eq:Gammarelation}
\Gamma(\lambda+a)=(\lambda+(a-1))\Gamma(\lambda+(a-1)), \;\;\; a\in \mathbb{Z}, \lambda \in \mathfrak{a}_\mathbb{C}^*
\end{equation}
repeatedly, we find the following expression of the quotient of the $\mathbf{c}$-functions:
\begin{eqnarray} \label{eq:cncm}
\frac{\mathbf{c}_n(\lambda)}{\mathbf{c}_m(\lambda)} 
&=& \frac{\Gamma(\lambda+\frac{1+m}{2})\Gamma(\lambda+\frac{1-m}{2})}{\Gamma(\lambda+\frac{1+n}{2})\Gamma(\lambda+\frac{1-n}{2})} \nonumber \\
&=&
\begin{cases}
1, & \text{ for } \vert n\vert=\vert m\vert, \\
\frac{(\lambda-\frac{\vert n\vert-1}{2})(\lambda-\frac{\vert n\vert-3}{2})\cdots(\lambda-\frac{\vert m\vert+1}{2})}{(\lambda+\frac{\vert n\vert-1}{2})(\lambda+\frac{\vert n\vert-3}{2})\cdots(\lambda+\frac{\vert m\vert+1}{2})}, & \text{ for } \vert n\vert>\vert m\vert,\\
\frac{(\lambda+\frac{\vert m\vert-1}{2})(\lambda+\frac{\vert m\vert-3}{2})\cdots(\lambda+\frac{\vert n\vert+1}{2})}{(\lambda-\frac{\vert m\vert-1}{2})(\lambda-\frac{\vert m\vert-3}{2})\cdots(\lambda-\frac{\vert n\vert+1}{2})}, & \text{ for } \vert n\vert<\vert m\vert.
\end{cases}
\end{eqnarray}
Note that the quotient has zeros in $\{\frac{\vert n\vert-1}{2},\frac{\vert n\vert-3}{2},\dots,\frac{\vert m\vert+1}{2}\}$ and poles in\\
$\{-\frac{\vert n\vert-1}{2},-\frac{\vert n\vert-3}{2}, \dots, -\frac{\vert m\vert+1}{2}\}$ for $\vert n\vert>\vert m\vert$, and inversely for $\vert n\vert<\vert m\vert$.
We know by (\ref{eq:normKSintwop3}) that the matrix coefficient of the Knapp-Stein intertwining operator
$ J_{w,\pm,\lambda} : H^{\pm,\lambda}_\infty \longrightarrow H^{\pm,-\lambda}_\infty$
with respect to the Fourier decomposition of $H^{\pm,\lambda}_\infty \cong \bigoplus_{n \text{ even or odd}} \delta_n$ is given by $\mathbf{c}_n(\lambda)$ (up to sign).

\begin{thm}[Intertwining conditions in (\textcolor{blue}{Level 1})] \label{thm:Meta1}
For $r>0$, let $A$ be the space of all
$\phi \in \prod_{\sigma\in \widehat{M}} \emph \Hol( \mathfrak{a}_\mathbb{C}^*,  \emph \End(H^\pm_\infty))$ such that $\phi$ statisfies the corresponding growth condition as well as the two intertwining conditions (\ref{eq:intwcondKS}) and
\begin{itemize}
\item[(D.b')] $\phi$ leaves every proper closed $G$-submodule $W_{\lambda}$ of $H^\pm_\infty$, listed in Remark~\ref{req:Wsubmodule}, invariant.
\end{itemize}
Then, $\mathcal{A}$ satisfies the conditions of Thm.~\ref{thm:Meta}, this means that
$\mathcal{A}=PW_r(G).$
\end{thm}

\begin{proof}
Note first, that the space $\mathcal{A}$ is $K\times K$-invariant and closed.
We have that $(D.b)$ of Thm.~\ref{thm:Meta} gives a condition for each $\sigma \in \widehat{M}= \{\pm\}$.
Let us first consider $\sigma=\{+\} \in \widehat{M}.$
By Example~\ref{exmp:SL2RM}, we have $\vert+\vert=0$ and $\vert\pi\vert=\vert k \vert+1>0$.
Now let  $\phi \in \mathcal{A}$ satisfying the assumption $(D.b)~(ii)$, i.e., in particular
\begin{equation} \label{eq:Assump}
\phi^{(0)}\Big(+, \frac{k}{2}\Big)\Big\vert_{D_{-k} \oplus D_k}=0, \;\;\; k \in 2\mathbb{Z}+1.
\end{equation}
Let us check that:
\begin{itemize}
\item[(a)] for $\text{Re}(\lambda) >0$, the intertwining operator $J_{-,+,\lambda}$ has zeros of order at most one.
\item[(b)] the kernel of $J_{w,+,\lambda}$ is equal to $0$ or $D_{-k} \oplus D_k$ for $\text{Re}(\lambda) >0$.
\end{itemize}
Consider the $K$-type $n\in 2\mathbb{Z}$ and the Harish-Chandra $\mathbf{c}$-function $\mathbf{c}_n$ as in (\ref{eq:cnSL(2,R)}).
If $n=0$, then $\mathbf{c}_0(\lambda)=\frac{1}{\sqrt{\pi}} \frac{\Gamma(\lambda)}{\Gamma(\lambda+\frac{1}{2})}$ and we see that $\mathbf{c}_0(\lambda)$ has no zeros and no poles for $\text{Re}(\lambda) >0$.
Thus, instead of $\mathbf{c}_n$, we can consider the quotient
$$\frac{\mathbf{c}_n(\lambda)}{\mathbf{c}_0(\lambda)} 
= \frac{\Gamma(\lambda+\frac{1}{2})^2}{\Gamma(\lambda+\frac{1+n}{2})\Gamma(\lambda+\frac{1-n}{2})}
=\frac{(\lambda-\frac{\vert n\vert-1}{2}) \cdots (\lambda-\frac{1}{2})}{(\lambda+\frac{\vert n\vert-1}{2}) \cdots (\lambda+\frac{1}{2})}.$$
It has zeros $\lambda \in \{\frac{1}{2}, \cdots, \frac{\vert n\vert-1}{2}\}$ of first order.
Due to (\ref{eq:normKSintw}), we know that the intertwining operator $J_{w,+,\lambda}$ is in relation with the $\mathbf{c}$-function.
If on all $K$-types, we have zeros of first order, then $J_{w,+,\lambda}$ should also have zeros of first order.
Hence $J_{w,+,\lambda}$ has zeros of at most order one, this proves the first assertation (a) of the claim.

Concerning (b), we need to check for which $K$-type $n$, the quotient $\frac{\mathbf{c}_n(\lambda)}{\mathbf{c}_0(\lambda)}$ has a zero, for fixed $\text{Re}(\lambda) >0$.
It is clear that, if $\lambda \notin I_+$, then $\frac{\mathbf{c}_n(\lambda)}{\mathbf{c}_0(\lambda)}$ has no zeros, i.e. that $\text{Ker}(J_{w,+,\lambda})=0$.
For fixed $\lambda=\frac{k}{2}, k\in 2\mathbb{Z}+1$, the $\mathbf{c}$-quotient $\frac{\mathbf{c}_n(\lambda)}{\mathbf{c}_0(\lambda)}$ has zeros if, and only if, $n$ is a $K$-type of $D_{-k}$ and $D_k$, i.e.
$\text{Ker}(J_{w,+,\lambda})=D_{-k} \oplus D_k$.
Thus, this implies (b).\\

By (b) and (\ref{eq:Assump}), we have that the operator $\phi^{(0)}(+, \lambda)$ annihiliates $\text{Ker}(J_{w,+,\lambda})$ for $\sigma=\{+\} \in \widehat{M}$ and $\text{Re}(\lambda)>0$.
By (a), we have that the order $n$ is equal to one, thus this condition is sufficient.\\
By arguing in a similar way as above for $\sigma=\{-\}\in \widehat{M}$ with $\vert-\vert=1$, and $\pi=D_k \in \widehat{G}_d$, $k \in 2\mathbb{Z}\backslash \{0\}$, with $\vert\pi\vert=\vert k\vert+1$, we can conclude that $\mathcal{A}$ satisfies the condition $(D.b)$ of Thm.~\ref{thm:Meta}.
\end{proof}

Now let us move to (\textcolor{blue}{Level 2}).

\begin{thm}[Intertwining conditions in (\textcolor{blue}{Level 2})] \label{thm:Meta2SL2R}
Let $m\in \mathbb{Z}$ be a $K$-type,
Then, $\psi \in \emph \Hol( \mathfrak{a}^*_\mathbb{C},H^{m\vert_M}_\infty)$ satisfies the intertwining condition (D.2) of Def.~\ref{defn:intcondL2L3} if and only if
\begin{itemize}
\item[(2.a)] $J_{w,m,\lambda} \psi(\lambda, \cdot)=  \mathbf{c}_{m}(\lambda)\psi(-\lambda, \cdot)$ for all $\lambda \in \mathfrak{a}^*_\mathbb{C},$
\item[(2.b)] $\psi(\lambda, \cdot) \in W_\lambda$, $\lambda \in I_\pm$, where $W_\lambda$ is the invariant $G$-submodule of $H^{\pm,\lambda}_\infty$ represented by the blue boxes in Fig.~\ref{fig:BoxSL2R}. Here the choice of $\{\pm\}$ depends on the parity of $m$.
\end{itemize}
\end{thm}
Notice that, if $W_\lambda$ is the whole colored blue box, then there is no intertwining conditions.

\begin{figure}[h!]
\begin{itemize}
\item for $m=0$:
\end{itemize}
\begin{center}
\begin{tikzpicture}[scale=0.5]
\draw (-1.5,0.5) node[anchor=south] {.};
\draw (-1,0.5) node[anchor=south] {.};
\draw (-0.5,0.5) node[anchor=south] {.};

\filldraw[fill=blue!80!white, draw=black](0,0) rectangle (2,1);
\draw (0,1) rectangle (1,2);
\draw (1,1) rectangle (2,2);

\draw (1,-0.7) -- (1,-0.3)  node[anchor=north] {};
\draw (1,-0.7) -- (1,-0.7)  node[anchor=north] {${-\frac{3}{2}}$};

\filldraw[fill=blue!80!white, draw=black](3,0) rectangle (5,1);
\draw (3,1) rectangle (4,2);
\draw (4,1) rectangle (5,2);

\draw (4,-0.7) -- (4,-0.3) node[anchor=north] {};
\draw (4,-0.7) -- (4,-0.7) node[anchor=north] {${-\frac{1}{2}}$};

\filldraw[fill=blue!80!white, draw=black](6,0) rectangle (7,1);
\filldraw[fill=blue!80!white, draw=black] (7,0) rectangle (8,1);
\filldraw[fill=blue!80!white, draw=black] (6,1) rectangle (8,2);

\draw (7,-0.7) -- (7,-0.3) node[anchor=north] {};
\draw (7,-0.7) -- (7,-0.7) node[anchor=north] {${\frac{1}{2}}$};

\filldraw[fill=blue!80!white, draw=black](9,0) rectangle (10,1);
\filldraw[fill=blue!80!white, draw=black] (10,0) rectangle (11,1);
\filldraw[fill=blue!80!white, draw=black] (9,1) rectangle (11,2);

\draw (10,-0.7) -- (10,-0.3) node[anchor=north] {};
\draw (10,-0.7) -- (10,-0.7) node[anchor=north] {${\frac{3}{2}}$};

\filldraw[fill=blue!80!white, draw=black] (12,0) rectangle (13,1);
\filldraw[fill=blue!80!white, draw=black] (13,0) rectangle (14,1);
\filldraw[fill=blue!80!white, draw=black] (12,1) rectangle (14,2);

\draw (13,-0.7) -- (13,-0.3)  node[anchor=north] {};
\draw (13,-0.7) -- (13,-0.7)  node[anchor=north] {${\frac{5}{2}}$};

\draw (14.5,0.5) node[anchor=south] {.};
\draw (15,0.5) node[anchor=south] {.};
\draw (15.5,0.5) node[anchor=south] {.};

\draw[thick,<->] (-2,-0.5) -- (16,-0.5) node[anchor=north west] {$\lambda$};
\end{tikzpicture}
\end{center}

\begin{itemize}
\item for $m \in 2\Z$:
\end{itemize}
\begin{itemize}
\item[$\cdot$] $m>0$:
\end{itemize}
\begin{center}
\begin{tikzpicture}[scale=0.5]
\draw (-1.5,0.5) node[anchor=south] {.};
\draw (-1,0.5) node[anchor=south] {.};
\draw (-0.5,0.5) node[anchor=south] {.};

\filldraw[fill=blue!80!white, draw=black](0,0) rectangle (2,1);
\draw (0,1) rectangle (1,2);
\draw (1,1) rectangle (2,2);

\draw (1,-0.7) -- (1,-0.3)  node[anchor=north] {};
\draw (1,-0.7) -- (1,-0.7)  node[anchor=north] {${-\frac{m+3}{2}}$};

\filldraw[fill=blue!80!white, draw=black](3,0) rectangle (5,1);
\draw (3,1) rectangle (4,2);
\draw (4,1) rectangle (5,2);

\draw (4,-0.7) -- (4,-0.3) node[anchor=north] {};
\draw (4,-0.7) -- (4,-0.7) node[anchor=north] {${-\frac{m+1}{2}}$};

\filldraw[fill=blue!80!white, draw=black](6,0) rectangle (8,1);
\draw (6,1) rectangle (7,2);
\filldraw[fill=blue!80!white, draw=black] (7,1) rectangle (8,2);

\draw (7,-0.7) -- (7,-0.3) node[anchor=north] {};
\draw (7,-0.7) -- (7,-0.7) node[anchor=north] {${-\frac{m-1}{2}}$};

\draw (8.5,0.5) node[anchor=south] {.};
\draw (9,0.5) node[anchor=south] {.};
\draw (9.5,0.5) node[anchor=south] {.};

\filldraw[fill=blue!80!white, draw=black](10,0) rectangle (12,1);
\draw (10,1) rectangle (11,2);
\filldraw[fill=blue!80!white, draw=black] (11,1) rectangle (12,2);

\draw (11,-0.7) -- (11,-0.3) node[anchor=north] {};
\draw (11,-0.7) -- (11,-0.7) node[anchor=north] {${-\frac{1}{2}}$};

\draw (13,0) rectangle (14,1);
\filldraw[fill=blue!80!white, draw=black] (14,0) rectangle (15,1);
\draw (13,1) rectangle (15,2);

\draw (14,-0.7) -- (14,-0.3)  node[anchor=north] {};
\draw (14,-0.7) -- (14,-0.7)  node[anchor=north] {${\frac{1}{2}}$};

\draw (15.5,0.5) node[anchor=south] {.};
\draw (16,0.5) node[anchor=south] {.};
\draw (16.5,0.5) node[anchor=south] {.};

\draw (17,0) rectangle (18,1);
\filldraw[fill=blue!80!white, draw=black] (18,0) rectangle (19,1);
\draw (17,1) rectangle (19,2);

\draw (18,-0.7) -- (18,-0.3)  node[anchor=north] {};
\draw (18,-0.7) -- (18,-0.7)  node[anchor=north] {${\frac{m-1}{2}}$};

\filldraw[fill=blue!80!white, draw=black] (20,0) rectangle (21,1);
\filldraw[fill=blue!80!white, draw=black] (21,0) rectangle (22,1);
\filldraw[fill=blue!80!white, draw=black] (20,1) rectangle (22,2);

\draw (21,-0.7) -- (21,-0.3)  node[anchor=north] {};
\draw (21,-0.7) -- (21,-0.7)  node[anchor=north] {${\frac{m+1}{2}}$};

\filldraw[fill=blue!80!white, draw=black] (23,0) rectangle (24,1);
\filldraw[fill=blue!80!white, draw=black] (24,0) rectangle (25,1);
\filldraw[fill=blue!80!white, draw=black] (23,1) rectangle (25,2);

\draw (24,-0.7) -- (24,-0.3) node[anchor=north] {};
\draw (24,-0.7) -- (24,-0.7) node[anchor=north] {${\frac{m+3}{2}}$};

\draw (25.5,0.5) node[anchor=south] {.};
\draw (26,0.5) node[anchor=south] {.};
\draw (26.5,0.5) node[anchor=south] {.};

\draw[thick,<->] (-2,-0.5) -- (27,-0.5) node[anchor=north west] {$\lambda$};
\end{tikzpicture}
\end{center}

\begin{itemize}
\item[$\cdot$] $m<0$:
\end{itemize}
\begin{center}
\begin{tikzpicture}[scale=0.5]
\draw (-1.5,0.5) node[anchor=south] {.};
\draw (-1,0.5) node[anchor=south] {.};
\draw (-0.5,0.5) node[anchor=south] {.};

\filldraw[fill=blue!80!white, draw=black](0,0) rectangle (2,1);
\draw (0,1) rectangle (1,2);
\draw (1,1) rectangle (2,2);

\draw (1,-0.7) -- (1,-0.3)  node[anchor=north] {};
\draw (1,-0.7) -- (1,-0.7)  node[anchor=north] {${-\frac{m+3}{2}}$};

\filldraw[fill=blue!80!white, draw=black](3,0) rectangle (5,1);
\draw (3,1) rectangle (4,2);
\draw (4,1) rectangle (5,2);

\draw (4,-0.7) -- (4,-0.3) node[anchor=north] {};
\draw (4,-0.7) -- (4,-0.7) node[anchor=north] {${-\frac{m+1}{2}}$};

\filldraw[fill=blue!80!white, draw=black](6,0) rectangle (8,1);
\filldraw[fill=blue!80!white, draw=black] (6,1) rectangle (7,2);
\draw (7,1) rectangle (8,2);

\draw (7,-0.7) -- (7,-0.3) node[anchor=north] {};
\draw (7,-0.7) -- (7,-0.7) node[anchor=north] {${-\frac{m-1}{2}}$};

\draw (8.5,0.5) node[anchor=south] {.};
\draw (9,0.5) node[anchor=south] {.};
\draw (9.5,0.5) node[anchor=south] {.};

\filldraw[fill=blue!80!white, draw=black](10,0) rectangle (12,1);
\filldraw[fill=blue!80!white, draw=black] (10,1) rectangle (11,2);
\draw (11,1) rectangle (12,2);

\draw (11,-0.7) -- (11,-0.3) node[anchor=north] {};
\draw (11,-0.7) -- (11,-0.7) node[anchor=north] {${-\frac{1}{2}}$};

\filldraw[fill=blue!80!white, draw=black] (13,0) rectangle (14,1);
\draw (14,0) rectangle (15,1);
\draw (13,1) rectangle (15,2);

\draw (14,-0.7) -- (14,-0.3)  node[anchor=north] {};
\draw (14,-0.7) -- (14,-0.7)  node[anchor=north] {${\frac{1}{2}}$};

\draw (15.5,0.5) node[anchor=south] {.};
\draw (16,0.5) node[anchor=south] {.};
\draw (16.5,0.5) node[anchor=south] {.};

\filldraw[fill=blue!80!white, draw=black](17,0) rectangle (18,1);
\draw (18,0) rectangle (19,1);
\draw (17,1) rectangle (19,2);

\draw (18,-0.7) -- (18,-0.3)  node[anchor=north] {};
\draw (18,-0.7) -- (18,-0.7)  node[anchor=north] {${\frac{m-1}{2}}$};

\filldraw[fill=blue!80!white, draw=black] (20,0) rectangle (21,1);
\filldraw[fill=blue!80!white, draw=black] (21,0) rectangle (22,1);
\filldraw[fill=blue!80!white, draw=black] (20,1) rectangle (22,2);

\draw (21,-0.7) -- (21,-0.3)  node[anchor=north] {};
\draw (21,-0.7) -- (21,-0.7)  node[anchor=north] {${\frac{m+1}{2}}$};

\filldraw[fill=blue!80!white, draw=black] (23,0) rectangle (24,1);
\filldraw[fill=blue!80!white, draw=black] (24,0) rectangle (25,1);
\filldraw[fill=blue!80!white, draw=black] (23,1) rectangle (25,2);

\draw (24,-0.7) -- (24,-0.3) node[anchor=north] {};
\draw (24,-0.7) -- (24,-0.7) node[anchor=north] {${\frac{m+3}{2}}$};

\draw (25.5,0.5) node[anchor=south] {.};
\draw (26,0.5) node[anchor=south] {.};
\draw (26.5,0.5) node[anchor=south] {.};

\draw[thick,<->] (-2,-0.5) -- (27,-0.5) node[anchor=north west] {$\lambda$};
\end{tikzpicture}
\end{center}

\begin{itemize}
\item for $m\in 2\Z+1$:
\end{itemize}
\begin{itemize}
\item[$\cdot$] $m>0$:
\end{itemize}

\begin{center}
\begin{tikzpicture}[scale=0.5]
\draw (-1.5,0.5) node[anchor=south] {.};
\draw (-1,0.5) node[anchor=south] {.};
\draw (-0.5,0.5) node[anchor=south] {.};

\filldraw[fill=blue!80!white, draw=black](0,0) rectangle (2,1);
\draw (0,1) rectangle (1,2);
\draw (1,1) rectangle (2,2);

\draw (1,-0.7) -- (1,-0.3)  node[anchor=north] {};
\draw (1,-0.7) -- (1,-0.7)  node[anchor=north] {${-\frac{m+3}{2}}$};

\filldraw[fill=blue!80!white, draw=black](3,0) rectangle (5,1);
\draw (3,1) rectangle (4,2);
\draw (4,1) rectangle (5,2);

\draw (4,-0.7) -- (4,-0.3) node[anchor=north] {};
\draw (4,-0.7) -- (4,-0.7) node[anchor=north] {${-\frac{m+1}{2}}$};

\filldraw[fill=blue!80!white, draw=black](6,0) rectangle (8,1);
						   \draw (6,1) rectangle (7,2);
\filldraw[fill=blue!80!white, draw=black] (7,1) rectangle (8,2);

\draw (7,-0.7) -- (7,-0.3) node[anchor=north] {};
\draw (7,-0.7) -- (7,-0.7) node[anchor=north] {${-\frac{m-1}{2}}$};

\draw (8.5,0.5) node[anchor=south] {.};
\draw (9,0.5) node[anchor=south] {.};
\draw (9.5,0.5) node[anchor=south] {.};

\filldraw[fill=blue!80!white, draw=black](10,0) rectangle (12,1);
\draw (10,1) rectangle (11,2);
\filldraw[fill=blue!80!white, draw=black] (11,1) rectangle (12,2);

\draw (11,-0.7) -- (11,-0.3) node[anchor=north] {};
\draw (11,-0.7) -- (11,-0.7) node[anchor=north] {$-1$};

\draw (13,0) rectangle (14,1);
\filldraw[fill=blue!80!white, draw=black] (14,0) rectangle (15,1);

\draw (14,-0.7) -- (14,-0.3) node[anchor=north] {};
\draw (14,-0.7) -- (14,-0.7) node[anchor=north] {$0$};

\draw (16,0) rectangle (17,1);
\filldraw[fill=blue!80!white, draw=black] (17,0) rectangle (18,1);
\draw (16,1) rectangle (18,2);

\draw (17,-0.7) -- (17,-0.3)  node[anchor=north] {};
\draw (17,-0.7) -- (17,-0.7)  node[anchor=north] {$1$};

\draw (18.5,0.5) node[anchor=south] {.};
\draw (19,0.5) node[anchor=south] {.};
\draw (19.5,0.5) node[anchor=south] {.};

\draw (20,0) rectangle (21,1);
\filldraw[fill=blue!80!white, draw=black] (21,0) rectangle (22,1);
\draw (20,1) rectangle (22,2);

\draw (21,-0.7) -- (21,-0.3)  node[anchor=north] {};
\draw (21,-0.7) -- (21,-0.7)  node[anchor=north] {${\frac{m-1}{2}}$};

\filldraw[fill=blue!80!white, draw=black] (23,0) rectangle (24,1);
\filldraw[fill=blue!80!white, draw=black] (24,0) rectangle (25,1);
\filldraw[fill=blue!80!white, draw=black] (23,1) rectangle (25,2);

\draw (24,-0.7) -- (24,-0.3)  node[anchor=north] {};
\draw (24,-0.7) -- (24,-0.7)  node[anchor=north] {${\frac{m+1}{2}}$};

\filldraw[fill=blue!80!white, draw=black] (26,0) rectangle (27,1);
\filldraw[fill=blue!80!white, draw=black] (27,0) rectangle (28,1);
\filldraw[fill=blue!80!white, draw=black] (26,1) rectangle (28,2);

\draw (27,-0.7) -- (27,-0.3) node[anchor=north] {};
\draw (27,-0.7) -- (27,-0.7) node[anchor=north] {${\frac{m+3}{2}}$};

\draw (28.5,0.5) node[anchor=south] {.};
\draw (29,0.5) node[anchor=south] {.};
\draw (29.5,0.5) node[anchor=south] {.};

\draw[thick,<->] (-2,-0.5) -- (30,-0.5) node[anchor=north west] {$\lambda$};
\end{tikzpicture}
\end{center}


\begin{itemize}
\item[$\cdot$] $m<0$:
\end{itemize}

\begin{center}
\begin{tikzpicture}[scale=0.5]
\draw (-1.5,0.5) node[anchor=south] {.};
\draw (-1,0.5) node[anchor=south] {.};
\draw (-0.5,0.5) node[anchor=south] {.};

\filldraw[fill=blue!80!white, draw=black](0,0) rectangle (2,1);
\draw (0,1) rectangle (1,2);
\draw (1,1) rectangle (2,2);

\draw (1,-0.7) -- (1,-0.3)  node[anchor=north] {};
\draw (1,-0.7) -- (1,-0.7)  node[anchor=north] {${-\frac{m+3}{2}}$};

\filldraw[fill=blue!80!white, draw=black](3,0) rectangle (5,1);
\draw (3,1) rectangle (4,2);
\draw (4,1) rectangle (5,2);

\draw (4,-0.7) -- (4,-0.3) node[anchor=north] {};
\draw (4,-0.7) -- (4,-0.7) node[anchor=north] {${-\frac{m+1}{2}}$};

\filldraw[fill=blue!80!white, draw=black](6,0) rectangle (8,1);
\filldraw[fill=blue!80!white, draw=black] (6,1) rectangle (7,2);
\draw (7,1) rectangle (8,2);

\draw (7,-0.7) -- (7,-0.3) node[anchor=north] {};
\draw (7,-0.7) -- (7,-0.7) node[anchor=north] {${-\frac{m-1}{2}}$};

\draw (8.5,0.5) node[anchor=south] {.};
\draw (9,0.5) node[anchor=south] {.};
\draw (9.5,0.5) node[anchor=south] {.};

\filldraw[fill=blue!80!white, draw=black](10,0) rectangle (12,1);
\filldraw[fill=blue!80!white, draw=black] (10,1) rectangle (11,2);
\draw (11,1) rectangle (12,2);

\draw (11,-0.7) -- (11,-0.3) node[anchor=north] {};
\draw (11,-0.7) -- (11,-0.7) node[anchor=north] {$-1$};

\filldraw[fill=blue!80!white, draw=black] (13,0) rectangle (14,1);
\draw (14,0) rectangle (15,1);

\draw (14,-0.7) -- (14,-0.3) node[anchor=north] {};
\draw (14,-0.7) -- (14,-0.7) node[anchor=north] {$0$};

\filldraw[fill=blue!80!white, draw=black] (16,0) rectangle (17,1);
\draw (17,0) rectangle (18,1);
\draw (16,1) rectangle (18,2);

\draw (17,-0.7) -- (17,-0.3)  node[anchor=north] {};
\draw (17,-0.7) -- (17,-0.7)  node[anchor=north] {$1$};

\draw (18.5,0.5) node[anchor=south] {.};
\draw (19,0.5) node[anchor=south] {.};
\draw (19.5,0.5) node[anchor=south] {.};

\filldraw[fill=blue!80!white, draw=black] (20,0) rectangle (21,1);
\draw (21,0) rectangle (22,1);
\draw (20,1) rectangle (22,2);

\draw (21,-0.7) -- (21,-0.3)  node[anchor=north] {};
\draw (21,-0.7) -- (21,-0.7)  node[anchor=north] {${\frac{m-1}{2}}$};

\filldraw[fill=blue!80!white, draw=black] (23,0) rectangle (24,1);
\filldraw[fill=blue!80!white, draw=black] (24,0) rectangle (25,1);
\filldraw[fill=blue!80!white, draw=black] (23,1) rectangle (25,2);

\draw (24,-0.7) -- (24,-0.3)  node[anchor=north] {};
\draw (24,-0.7) -- (24,-0.7)  node[anchor=north] {${\frac{m+1}{2}}$};

\filldraw[fill=blue!80!white, draw=black] (26,0) rectangle (27,1);
\filldraw[fill=blue!80!white, draw=black] (27,0) rectangle (28,1);
\filldraw[fill=blue!80!white, draw=black] (26,1) rectangle (28,2);

\draw (27,-0.7) -- (27,-0.3) node[anchor=north] {};
\draw (27,-0.7) -- (27,-0.7) node[anchor=north] {${\frac{m+3}{2}}$};

\draw (28.5,0.5) node[anchor=south] {.};
\draw (29,0.5) node[anchor=south] {.};
\draw (29.5,0.5) node[anchor=south] {.};

\draw[thick,<->] (-2,-0.5) -- (30,-0.5) node[anchor=north west] {$\lambda$};
\end{tikzpicture}
\end{center}
\caption{Boxes-pictures for $G=\SL(2,\R)$.}
\label{fig:BoxSL2R}
\end{figure}

\begin{proof}
We need to show that the conditions $(2.a)$ and $(2.b)$ correspond to the condition $(D.2)$ in Def.~\ref{defn:intcondL2L3}.
In fact, by Prop.~\ref{prop:normKSLevel2}(a), we have that $(2.a)$ is a special case of $(D.2)$.

Concerning $(2.b)$, condition $(D.2)$ says that for each $W$ we have an intertwining condition corresponding to $(D.b')$ in Thm.~\ref{thm:Meta1}.
Now we need to extract in which of these $W_\lambda$, there is an intertwining condition.
If the $K$-type $m$ is in a closed $G$-submodule $W_\lambda$ of $H^{\pm,\lambda}_\infty$, then $D^m_W$ is one-dimensional. 
Hence, by $(D.2)$, $\psi$ has values in this $G$-submodule $W_\lambda$. 
By $(D.b')$ in Thm.~\ref{thm:Meta1}, we thus take the smallest closed proper invariant $G$-submodule of them. 
Otherwise, if the $K$-type is not in a closed $G$-submodule $W_\lambda$ of $H^{\pm,\lambda}_\infty$, then $D^m_{W_\lambda} =\{0\}$ and thus there are no intertwining conditions.
Consequently, we obtain the boxes-pictures in Fig.~\ref{fig:BoxSL2R}.
\end{proof}

The final step will be to move to (\textcolor{blue}{Level 3}). 
Note that $\text{Hom}_M(E_n,E_m) = \{0\}$, if $n \equiv m$ (mod 2).

\begin{defn} \label{defn:qnm}
Let $n \equiv m$ (mod 2). 
We define the polynomial $q_{n,m}$ in $\lambda \in \mathfrak{a}^*_\mathbb{C}$ with values in $\emph \Hom_M(E_n,E_m) \cong \mathbb{C}$ by
\begin{equation} \label{eq:qnm}
q_{n,m}(\lambda):= \begin{cases}
1, & \text{ if } n=m\\
(\lambda + \frac{\vert m \vert+1}{2})(\lambda + \frac{\vert m \vert+3}{2}) \cdots (\lambda + \frac{\vert n \vert-1}{2}), & \text{ if } \vert n \vert>\vert m \vert \text{ and same signs}\\
(\lambda - \frac{\vert n \vert+1}{2})(\lambda - \frac{\vert n \vert+3}{2})  \cdots (\lambda - \frac{\vert m \vert-1}{2}), & \text{ if } \vert n \vert<\vert m \vert \text{ and same signs}\\
(\lambda + \frac{\vert n \vert-1}{2})(\lambda + \frac{\vert n \vert-3}{2})  \cdots (\lambda - \frac{\vert m \vert-1}{2}), & \text{ else, with different signs.}
\end{cases}
\end{equation}
\end{defn}

\begin{thm}[Intertwining conditions in (\textcolor{blue}{Level 3})] \label{thm:Meta3SL2R}
Let $n \equiv m$ (mod 2) be two $K$-types. 
Then, $\varphi \in \emph \Hol(\mathfrak{a}^*_\mathbb{C},\emph \Hom_M(E_n,E_m))$ satisfies the intertwining condition (D.3) of Def.~\ref{defn:intcondL2L3} if, and only if, there exists an even holomorphic function $h\in \emph \Hol(\lambda^2)$ such that 

\begin{equation} \label{eq:intcondLevel3}
\varphi(\lambda)= h(\lambda) \cdot q_{n,m}(\lambda), \;\;\; \lambda \in \mathfrak{a}^*_\mathbb{C},
\end{equation}
where
$q_{n,m}$ is the polynomial (\ref{eq:qnm}).
\end{thm}

\begin{proof}
By Thm.~\ref{thm:Meta2SL2R}, it is sufficient to prove that the conditions $(2.a)$ and $(2.b)$ correspond to (\ref{eq:intcondLevel3}).
In particular, we want to show that $(\lambda,k_\theta) \mapsto \varphi(\lambda)e^{in\theta}$ satisfies $(2.b)$ if, and only if, $\varphi$ has zeros at the zeros of the polynomial $q_{n,m}$.
From Thm.~\ref{thm:Meta2SL2R} $(2.b)$, we know that the invariant $G$-submodules $W_\lambda$ are represented by the boxes-pictures in Fig.~\ref{fig:BoxSL2R}.
Thus, we need to check, where the $K$-type $n$ is not in the colored blue invariant $G$-submodule $W_\lambda$.
We leave it to the reader to check that this happens exactly at the zeros of $q_{n,m}$.
Thus, we can deduce that $\varphi$ is of the form (\ref{eq:intcondLevel3}) with $h$ an arbitrary holomorphic function.

Concerning the correspondence between the conditions $(2.a)$ and (\ref{eq:intcondLevel3}), by Prop.~\ref{prop:normKSLevel2}(b), we observe that  $(2.a)$ corresponds to $(D.3)$:
\begin{equation} \label{eq:cnmqnm}
(-1)^{(m-n)/2}\frac{\mathbf{c}_n(\lambda)}{\mathbf{c}_m(\lambda)} \varphi(\lambda)=\varphi(-\lambda), \;\; \lambda \in \mathfrak{a}^*_\mathbb{C}, n,m \in \mathbb{Z}.
\end{equation}
By using Def.~\ref{defn:qnm}, we observe that
$\frac{q_{n,m}(-\lambda)}{q_{n,m}(\lambda)}= (-1)^{(m-n)/2}\frac{\mathbf{c}_n(\lambda)}{\mathbf{c}_m(\lambda)},$ for $\lambda \in \mathfrak{a}^*_\mathbb{C}, n,m \in \mathbb{Z}.$
Hence, we obtain
$$\frac{\varphi(\lambda)}{q_{n,m}(\lambda)}=\frac{\varphi(-\lambda)}{q_{n,m}(-\lambda)}.$$
This means that (\ref{eq:cnmqnm}) is satisfied if and only if $h(\lambda)=h(-\lambda)$ for $\lambda \in \mathfrak{a}^*_\mathbb{C}$.
\end{proof}

We now have completely determined the Paley-Wiener-(Schwartz) spaces for $G=\mathbf{SL}(2,\mathbb{R})$ in (\textcolor{blue}{Level 2}) and (\textcolor{blue}{Level 3}).


\subsection*{The case $G=\SL(2,\R)\times \SL(2,\R)$} \label{sect:SL4R}
Now let
$$G:=G'\times G' = \mathbf{SL}(2,\mathbb{R})\times \mathbf{SL}(2,\mathbb{R}).$$ 
Since $K'=\mathbf{SO}(2)$ is a maximal compact subgroup of $G'$, $K:=K'\times K'$ is maximal compact in $G$.

The irreducible representations of $K$ are given by pairs $(n_1,n_2), n_i \in \widehat{K}' \cong \mathbb{Z}$ \cite[Sect. 2.36]{Wallach1}.
More precisely, we denote by a tuple of integers 
$n:=(n_1,n_2) \in \mathbb{Z}\times\mathbb{Z} \cong \mathbb{Z}^2 \cong \widehat{K}$
the $K$-representions on the vector space $E_n:=E_{n_1} \otimes E_{n_2} \cong \mathbb{C}$ with action $[n_1,n_2](k_1,k_2)=n_1(k_1) \otimes n_2(k_2)$.
The associated homogeneous line bundle over $X:=X' \times X'$ is denoted by
$\mathbb{E}_n$.
For $l,n \in \widehat{K}$, we observe that $\text{Hom}_M(E_l,E_n) = \{0\}$, if $l_1 \not\equiv n_1$ (mod 2) or $l_2 \not\equiv n_2$ (mod 2).
Note that $\mathfrak{a}^*_\mathbb{C} \cong  \mathbb{C} \times \mathbb{C}$.
By using Def.~\ref{defn:qnm}, we define for $l,n \in \mathbb{Z}^2$
\begin{equation} \label{eq:ln}
l_1 \equiv n_1 \text{ (mod 2)}, \; l_2 \equiv n_2\text{ (mod 2)},
\end{equation}
the polynomial $q_{l,n}$ given by
\begin{equation} \label{eq:qnmdecomp}
q_{l,n}(\lambda_1,\lambda_2):=q_{l_1,n_1}(\lambda_1)\cdot q_{l_2,n_2}(\lambda_2), \;\;\;\; (\lambda_1,\lambda_2) \in \mathbb{C} \times \mathbb{C} \cong \mathfrak{a}^*_\mathbb{C},
\end{equation}
where $q_{l_i,n_i}, i=1,2,$ is the 'intertwining' polynomial (\ref{eq:qnm}).

\begin{thm}[Intertwining condition in (\textcolor{blue}{Level 3})] \label{thm:Meta3SL4R}
Let $l,n \in \mathbb{Z}^2$ be two tuples of integers satisfying (\ref{eq:ln}).
Then, $\varphi \in \emph \Hol(\mathfrak{a}^*_\mathbb{C},\emph \Hom_M(E_l,E_n))$ satisfies the intertwining condition (D.3) of Def.~\ref{defn:intcondL2L3} if and only if there exists an holomorphic function $h\in \emph \Hol(\lambda_1^2,\lambda_2^2)$, i.e.
$h(\lambda_1,\lambda_2)=h(-\lambda_1,\lambda_2)=h(\lambda_1,-\lambda_2),$
such that 
\begin{equation} \label{eq:compcondLevel3}
\varphi(\lambda_1,\lambda_2):=h(\lambda_1,\lambda_2) \cdot q_{l,n}(\lambda_1,\lambda_2).
\end{equation}
\end{thm}

To prove Thm.~\ref{thm:Meta3SL4R}, we first need a density argument, which permits us to approximate the even holomorphic function $h$ in (\ref{eq:compcondLevel3}) by even polynomials.

\begin{lem} \label{lem:densitySL4R}
Consider the subset $P$ of polynomial functions
in $\emph \Hol(\lambda_1^2,\lambda_2^2)$.
Then, $P\subset \emph \Hol(\lambda_1^2,\lambda_2^2)$ is dense with respect to uniform convergence  on compact subsets of $\mathbb{C}^2$.
\end{lem}

\begin{proof}
Let $h(\lambda_1,\lambda_2) \in \text{Hol}(\lambda_1^2,\lambda_2^2)$.
Consider the Taylor series 
at the point $0=(0,0)$ in two variables $(\lambda_1,\lambda_2)\in \mathfrak{a}^*_\mathbb{C}$:
$$\sum_{\alpha} a_\alpha \lambda^\alpha =\sum_{\alpha_1,\alpha_2} a_{\alpha_1,\alpha_2} \lambda_1^{\alpha_1}\lambda_2^{\alpha_2},$$
where $a_\alpha$ are constants and the sum runs over multi-indices $\alpha=(\alpha_1,\alpha_2), \alpha_j\in \mathbb{N}_0$. 
Let $\vert\alpha\vert= \alpha_1 + \alpha_2$.
Note that $a_{\alpha_1,\alpha_2}=0$, if $\alpha_1$ or $\alpha_2$ is odd.
Thus the Taylor polynomials $\sum_{\vert\alpha\vert \leq k} a_\alpha \lambda^\alpha$
belong to $P$.
The Taylor polynomials converge locally uniformly to $h$ for $k$ going to infinity. 
\end{proof}

Next, by using the Iwasawa decomposition of $g=(g_1,g_2) \in G$, the 'exponential' function $e^{n}_{\lambda,k}$ can be rewritten as follows.

\begin{prop} \label{prop:expdecomp}
For fixed $\lambda=(\lambda_1,\lambda_2) \in \mathfrak{a}^*_\mathbb{C}$ and $k=(k_1,k_2) \in K$, the function $e^{n}_{\lambda,k} \in C^\infty(G)$ defined as in Def.~\ref{defn:FTsect}, is a product of the corresponding functions on $G'$:
$$e^{n}_{\lambda,k}(g_1,g_2)=e^{n_1}_{\lambda_1,k_1}(g_1) \cdot e^{n_2}_{\lambda_2,k_2}(g_2), \;\;\;\; (g_1,g_2) \in G.$$
\end{prop}

\begin{proof}
Consider the Iwasawa decomposition of 
$$g=(g_1,g_2)=(n'_1a_1k'_1,n'_2a_2k'_2)=n'ak' \in G$$ so that $n'=(n'_1,n'_2) \in N, a=(a_1,a_2) \in A$ and $k'=(k'_1,k'_2) \in K.$
One can easily deduce that for $\lambda=(\lambda_1,\lambda_2) \in \mathfrak{a}^*_\mathbb{C}$, we get
$$e^{(\lambda+\rho)\log(a)}=e^{(\lambda_1+\rho)\log(a_1)+(\lambda_2+\rho)\log(a_2)}=a_1^{\lambda_1+\rho}\cdot a_2^{\lambda_2+\rho}.$$
Hence, for $g\in G$, we then have
\begin{eqnarray*}
e^{n}_{\lambda,k}(g)=e^{n}_{\lambda,k}(g_1,g_2)
= a_1^{\lambda_1+\rho}a_2^{\lambda_2+\rho} e^n_{\lambda,1}(k_1,k_2) 
&=& a_1^{\lambda_1+\rho}a_2^{\lambda_2+\rho} n_1(k_1)n_2(k_2) \\
&=& a_1^{\lambda_1+\rho}a_2^{\lambda_2+\rho} e^{n_1}_{\lambda_1,1}(k_1)e^{n_2}_{\lambda_2,1}(k_2) \\
&\stackrel{(\ref{eq:exptau})}{=}& e^{n_1}_{\lambda_1,k_1}(g_1) \cdot e^{n_2}_{\lambda_2,k_2}(g_2).
\end{eqnarray*}
\end{proof}

Let $\prescript{}{l}{PWS}_{n,H}(\mathfrak{a}^*_\mathbb{C}) = \{ \varphi \in \text{Hol}(\mathfrak{a}^*_\mathbb{C}, \text{Hom}_M(E_l,E_n)) \;\vert\; \varphi$ satisfies $(D.3)\}$ be the 'pre'-Paley-Wiener-Schwartz space.
Note that 
$$\prescript{}{l}{PW}_{n}(\mathfrak{a}^*_\mathbb{C}) \subset \prescript{}{l}{PWS}_{n}(\mathfrak{a}^*_\mathbb{C})  \subset \prescript{}{l}{PWS}_{n,H}(\mathfrak{a}^*_\mathbb{C}).$$

\begin{proof}[Proof of Thm.~\ref{thm:Meta3SL4R}]
It suffices to show that
\begin{itemize}
\item[(a)] every function $\varphi \in \prescript{}{l}{PW}_{n}(\mathfrak{a}^*_\mathbb{C})$ is of the form (\ref{eq:compcondLevel3}) and
\item[(b)] (inversely) if $\varphi$ is of the form (\ref{eq:compcondLevel3}), then it is in $\prescript{}{l}{PWS}_{n,H}(\mathfrak{a}^*_\mathbb{C})$.
\end{itemize}
Let $\varphi \in \prescript{}{l}{PW}_{n}(\mathfrak{a}^*_\mathbb{C})$.
By the Paley-Wiener Thm.~\ref{thm:PWsect}, there exists $f\in C^{\infty}_c(G,l,n)$ with
$\prescript{}{l}{\mathcal{F}}_{n} (f)=\varphi$.
By Fubini's theorem and Prop.~\ref{prop:expdecomp}, we have
\begin{eqnarray} \label{eq:FTn1}
\prescript{}{l}{\mathcal{F}}_{n} (f)(\lambda) = \int_{G'} e^{n_1}_{\lambda_1,1}(g_1) \Big( \int_{G'} e^{n_2}_{\lambda_2,1}(g_2) f(g_1,g_2) \;d g_2 \Big) \; d g_1 
&=& \int_{G'} e^{n_1}_{\lambda_1,1}(g_1) \tilde{f}_{\lambda_2}(g_1) \;d g_1 \nonumber \\
&=&\prescript{}{l_1}{\mathcal{F}}_{n_1} \tilde{f}_{\lambda_2}(\lambda_1),
\end{eqnarray}
where we set $\tilde{f}_{\lambda_2}(g_1) :=  \int_{G'} e^{n_2}_{\lambda_2,1}(g_2) f(g_1,g_2) \;d{g_2}.$
Note that $\tilde{f}_{\lambda_2} \in C^\infty_c(X',\mathbb{E}_{n_1}).$
Similarly, if we fix $\lambda_1 \in \mathfrak{a}^*_\mathbb{C}$, we have $\tilde{f}_{\lambda_1}(g_2) :=  \int_{G'} e^{n_1}_{\lambda_1,k_1}(g_1) f(g_1,g_2) \;d{g_1} \in  C^\infty_c(X',\mathbb{E}_{n_2})$
and
$\prescript{}{l}{\mathcal{F}}_{n} (f)(\lambda)=\prescript{}{l_2}{\mathcal{F}}_{n_2} \tilde{f}_{\lambda_1}(\lambda_2)$.
Thus, by Thm.~\ref{thm:Meta3SL2R}, the Fourier transform has the form
\begin{equation} \label{eq:FTn1n2}
\prescript{}{l}{\mathcal{F}}_{n} (f)(\lambda) = h_{\lambda_2}(\lambda_1) \cdot q_{l_1,n_1}= h_{\lambda_1}(\lambda_2) \cdot q_{l_2,n_2}(\lambda_2),
\end{equation}
where $h_{\lambda_i}$ is an even holomorphic function in 
$\lambda_i\in \mathfrak{a}^*_\mathbb{C}$.
In view of (\ref{eq:qnmdecomp}), we deduce that there exists $h \in \text{Hol}(\lambda_1^2, \lambda_2^2)$ such that
$$\varphi(\lambda)=\prescript{}{l}{\mathcal{F}}_{n}f(\lambda) = h(\lambda_1,\lambda_2) \cdot q_{l,n}(\lambda_1,\lambda_2), \;\;\; \lambda=(\lambda_1,\lambda_2) \in \mathfrak{a}^*_\mathbb{C},$$
as desired.

Concerning (b), let $\varphi$ of the form (\ref{eq:compcondLevel3}).
Then, for $(\lambda_1, \lambda_2) \in \mathfrak{a}^*_\mathbb{C}$
$$\varphi(\lambda_1,\lambda_2)=h(\lambda_1,\lambda_2) q_{l_1,n_1}(\lambda_1) q_{l_2,n_2}(\lambda_2).$$
By Lem.~\ref{lem:densitySL4R}, we can approximate $h$ by linear combinations of products of  monomials of the form $\lambda_1^{\alpha_1}$ and $\lambda_2^{\alpha_2}$, where $\alpha_1$ and $\alpha_2$ are even.
By Thm.~\ref{thm:Meta3SL2R} and the Paley-Wiener-Schwartz Thm.~\ref{thm:PWsect},
$$
\lambda_i^{\alpha_i} q_{l_i,n_i}(\lambda_i)= \prescript{}{l_i}{\mathcal{F}}_{n_i}(f_i)(\lambda_i) 
\in \prescript{}{l_i}{PWS}_{n_i}(\mathfrak{a}^*_\mathbb{C})
$$
for some distribution $f_i\in C^{-\infty}_c(G',l_i,n_i), i=1,2$.
Consider now the tensor product of these two distributions:
$f_1 \otimes f_2 \in C^{-\infty}_c(G,l,n).$
By using the computations involving Fubini's theorem from the beginning of the proof, we obtain that
$$\prescript{}{l}{\mathcal{F}}_n (f_1 \otimes f_2)(\lambda)= \prescript{}{l_1}{\mathcal{F}}_{n_1} (f_1)(\lambda_1) \cdot \prescript{}{l_2}{\mathcal{F}}_{n_2} (f_2)(\lambda_2) = \lambda^\alpha q_{l,n}(\lambda).$$
Hence $\lambda^\alpha q_{l,n} \in \prescript{}{l}{PWS}_{n}(\mathfrak{a}^*_\mathbb{C})\subset \prescript{}{l}{PWS}_{n,H}(\mathfrak{a}^*_\mathbb{C})$.
Since $\prescript{}{l}{PWS}_{n,H}(\mathfrak{a}^*_\mathbb{C}) \subset \text{Hol}(\mathfrak{a}^*_\mathbb{C})$ is closed with respect to uniform convergence on compact subsets, we conclude that $\varphi=h\cdot  q_{l,n} \in \prescript{}{l}{PWS}_{n,H}(\mathfrak{a}^*_\mathbb{C}).$
\end{proof}

By Thm.~\ref{thm:Meta3SL4R}, we have explicitly determined the Paley-Wiener(-Schwartz) spaces for $G=\mathbf{SL}(2,\mathbb{R}) \times \mathbf{SL}(2,\mathbb{R})$ in (\textcolor{blue}{Level 3}).\\
Moreover, all the previous results can be generalised to
$G=\mathbf{SL}(2,\mathbb{R})^d, d\geq 2$.

\section{The case $G=\SL(2,\C)$} \label{sect:SL2C}

Let 
$G=\mathbf{SL}(2,\mathbb{C})=\{g \in \mathbf{GL}(2,\mathbb{C}) \;\vert\; \det(g)=1\}$
 be the special linear group of $\mathbb{C}^2$ with maximal compact subgroup
$$K=\mathbf{SU}(2)=\Big\{ \begin{pmatrix} \alpha & \beta \\ - \overline{\beta} & \overline{\alpha} \end{pmatrix} \in  \mathbf{GL}(2,\mathbb{C}) \; \Big\vert \; \vert\alpha\vert^2+\vert\beta\vert^2=1 \Big\}.$$
Note that $K$ is homeomorphic to the 3-sphere and therefore simply connected. 
Furthermore, we can take $A \subset \mathbf{SL}(2,\mathbb{R}) \subset G$ as in Sect.~\ref{sect:SL2R} and
$$N:=\Big\{ \begin{pmatrix}1 & t \\ 0 & 1 \end{pmatrix} \;\Big\vert\; t \in \mathbb{C} \Big\}.$$
\noindent
We identify
$\mathfrak{a}^*_\mathbb{C} \cong \mathbb{C}$ by sending $\lambda$ to $\lambda(H)$, where
$H=
\begin{pmatrix} 1&0 \\ 0&-1 \end{pmatrix} 
\in \mathfrak{a}$.
Note that $\rho(H)=2.$
For the irreducible complex representations of $K$ we have (e.g. \cite[Sect. 5.7]{Wallach})
$$\widehat{K}=\{\delta_n \;\vert\; n \in \mathbb{N}_0\} \cong \mathbb{N}_0 \text{ with } d_{\delta_n}:=\dim(\delta_n)=n+1.$$
The tensor product of two irreducible $K$-representations decomposes into irreducibles according to the classical Clebsch-Gordan rule (e.g. \cite[5.7.1 (1)]{Wallach}):
\begin{equation} \label{eq:CGrule}
\delta_n \otimes \delta_m = \bigoplus_{0 \leq j \leq \min(n,m)} \delta_{n+m-2j}, \;\;\;\, n,m \in  \mathbb{N}_0.
\end{equation}
In addition,
$M=Z_K(A)=\Big\{m_\theta:= \begin{pmatrix} e^{i\theta} & 0 \\ 0 & e^{-i\theta} \end{pmatrix} \;\Big\vert\; \theta \in \mathbb{R} \Big\}$
 is abelian and a maximal torus in $K$.
We parametrize $\widehat{M}:=\{\sigma_l \;\vert\; l \in \mathbb{Z}\} \cong \mathbb{Z}$ by the integers with
$\sigma_l(m_\theta)=e^{il\theta} \in \mathbf{U}(1)$.
Moreover, let $\chi_n:M \rightarrow \mathbb{C}$ denote the character of the finite-dimensional irreducible representation $(\delta_n,E_n)$ of $K$.
Then, the Weyl character formula for $m_\theta \in M$ (e.g. \cite[Chap.~V.6]{Knapp1})
\begin{eqnarray*}
\chi_n(m_\theta)
= \frac{e^{i(n+1)\theta}-e^{-i(n+1)\theta}}{e^{i\theta}-e^{-i\theta}}=\frac{\sin((n+1)\theta)}{\sin(\theta)} 
=e^{-in\theta} + e^{-i(n-2)\theta} + \dots+ e^{in\theta}
\end{eqnarray*}
tells us that the weights of $(\delta_n,E_n)$ have the form $-n,-(n-2), \dots, n-2,n$, each with multiplicity one.
The following important result describes the reducibility of the principal series representations of $\mathbf{SL}(2,\mathbb{C})$. We refer for example to Wallach's book \cite[Sect. 5.7]{Wallach} for a proof.
Note that Wallach's proof is also valid for $G$-representation on smooth vectors (see remark before Thm.~\ref{thm:Bargmann}).

\begin{thm}[Structure of principal series representations of $\SL(2,\C)$] \label{thm:WallachSL2C}
The principal series representations $H^{\sigma,\lambda}_\infty$ of $\mathbf{SL}(2,\mathbb{C})$ is reducible if and only if $\lambda$ is real and
$$\vert\lambda\vert>\vert\sigma\vert,\; \vert\lambda\vert-\vert\sigma\vert \text{ even integer}.
$$
In this case, for $\lambda>0$, there is a unique irreducible subrepresentation $R^{\sigma,\lambda}$ of each $H^{\sigma,\lambda}_\infty$.
Then, we have
\begin{center}
\begin{tikzpicture}[scale=0.5]
\draw (-1.7,1.5) -- (-1.7,1.5) node[anchor=north] {$H^{-\sigma,-\lambda}_\infty=$};
\draw (0.5,0) rectangle (4.5,1) node[pos=.5] {$F_{m,n}$};
\draw (0.5,1) rectangle (4.5,2) node[pos=.5] {$R^{\sigma,\lambda}$};

\draw (7.5,1.5) -- (7.5,1.5) node[anchor=north] {$H^{\sigma,\lambda}_\infty=$};

\draw (9.5,0) rectangle (13.5,1) node[pos=.5] {$R^{\sigma,\lambda}$};
\draw (9.5,1) rectangle (13.5,2) node[pos=.5] {$F_{m,n}$};
\end{tikzpicture}
\end{center}
where $m=\frac{\sigma+\lambda}{2}-1, n=\frac{\lambda-\sigma}{2}-1$
and 
$F_{m,n}$ is an irreducible finite-dimensional $G$-representation that is isomorphic to $\delta_m \otimes \delta_n$ as a $K$-representation.
Moreover, there is an intertwining operator
$$L_{\sigma,\lambda}: H^{-\lambda,-\sigma}_\infty \longrightarrow H^{\sigma, \lambda}_\infty$$
so that $\emph \Ker(J_{w,\sigma,\lambda}) = \emph \Impart(L_{\sigma,\lambda})=R^{\sigma,\lambda}$.
In particular $R^{\sigma,\lambda}$ is isomorphic to $H^{-\lambda,-\sigma}_\infty$.
Here, $J_{w,\sigma,\lambda}: H^{\sigma,\lambda} \longrightarrow H^{-\sigma,-\lambda}$ denotes the Knapp-Stein intertwining operator defined in Def.~\ref{def:KSintop} with $w=-1$ and $m_w=\begin{pmatrix} 0 & -1 \\ 1 & 0 \end{pmatrix}.$
\end{thm}

Furthermore, from the intertwining operator
$L_{\sigma,\lambda}$, we can deduce the existence of further intertwining operators:
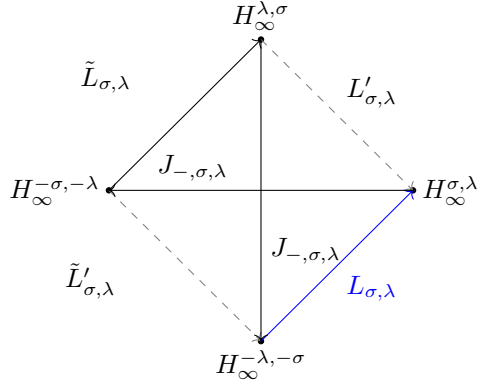
\begin{center}
\begin{tikzpicture}
\foreach \point in {(0,-2),(-2,0),(0,2),(2,0)}{\draw[black,fill=black] \point circle (1pt);};
 \node[anchor=south, black] at (0,2) {$H^{\lambda, \sigma}_\infty$};
\node[anchor=east, black] at (-2,0) {$H^{-\sigma,-\lambda}_\infty$};
\node[anchor=west, black] at (2,0) {$H^{\sigma,\lambda}_\infty$};
 \node[anchor=north, black] at (0,-2) {$H^{-\lambda,-\sigma}_\infty$};
\node[anchor=west, black] at (-2.5,1.5) {$\tilde{L}_{\sigma,\lambda}$};
\node[anchor=west, blue] at (1,-1.3) {$L_{\sigma,\lambda}$};
\node[anchor=west, black] at (1,1.3) {$L'_{\sigma,\lambda}$};
\node[anchor=west, black] at (-2.7,-1.2) {$\tilde{L}'_{\sigma,\lambda}$};
\node[anchor=south, black] at (-0.9,0) {$J_{-,\sigma,\lambda}$};
\node[anchor=west, black] at (0,-0.8) {$J_{-,\sigma,\lambda}$};
\draw[thin, black,->] (-2,0) -- (0,2); 
\draw[thin, blue,->] (0,-2) -- (2,0); 
\draw[thin, dashed, gray, ->] (-2,0) -- (0,-2); 
\draw[thin, dashed, gray, ->] (0,2) -- (2,0); 
\draw[thin, black,->] (0,2) -- (0,-2); 
\draw[thin, black,->] (2,0) -- (-2,0); 
\end{tikzpicture}
\end{center}

Note that the operators $L_{\sigma,\lambda}, \tilde{L}_{\sigma,\lambda}, L'_{\sigma,\lambda}$ and $\tilde{L}'_{\sigma,\lambda}$ are precisely the Zelobenko operators (also called BGG-operators) for $G=\mathbf{SL}(2,\mathbb{C})$ \cite{Zelo}, \cite{BGG1}, \cite{BGG2}.
In Fig.~\ref{fig:HSL(2,C)}, we illustrate the principal series representations (with regular integral infinitesimal character) in a grid, where
the horizontal axis represents the values of $\lambda \in \mathfrak{a}^*_\mathbb{C}$ and the vertical one the values of $\sigma \in \widehat{M}$. 
Note that inside the region $\{\pm \vert\sigma\vert > \lambda\}$, we have the irreducible principal series representations $H_\infty^{-\lambda,-\sigma}$ respectively $H_\infty^{\lambda, \sigma}$ colored in gray and outside the reducible ones, colored in black.
\begin{figure}[h!]
\begin{center}
\begin{tikzpicture}[scale=0.72]

\draw[thin,->] (-8,-5) -- (-8,-3) node[right] {$\sigma$}; 
\draw[thin,->] (-8,-5) -- (-6,-5) node[above] {$\lambda$}; 
\draw[thin,dashed,gray,-] (-5.5,-5.5) -- (5.5,5.5)  node[above] {$\sigma=\lambda$}; 
\draw[thin,dashed,gray,-] (5.5,-5.5) -- (-5.5,5.5)  node[above] {$\sigma=-\lambda$}; 

\foreach \point in {(2,0),(4,0),(3,1),(5,1),(4,2),(5,3),(3,-1),(5,-1),(4,-2),(5,-3)}{
    \fill \point circle (2pt);}
\foreach \point in {(-2,0),(-4,0),(-3,-1),(-5,-1),(-4,-2),(-5,-3),(-3,1),(-5,1),(-4,2),(-5,3)}{
   \fill \point circle (2pt);}
\foreach \point in {(0,2),(-1,3),(1,3),(0,4),(-2,4),(2,4),(-3,5),(-1,5),(1,5),(3,5)}{
   \draw[gray,fill=gray] \point circle (2pt);}
\foreach \point in {(0,-2),(1,-3),(-1,-3),(0,-4),(2,-4),(-2,-4),(3,-5),(1,-5),(-1,-5),(-3,-5)}{
   \draw[gray,fill=gray] \point circle (2pt);}

\foreach \point in {(4,-2),(-2,4),(-4,2),(2,-4)}{\draw[orange,fill=orange] \point circle (2pt);}

\foreach \point in {(0,-2),(-2,0),(0,2),(2,0)}{\draw[blue,fill=blue] \point circle (2pt);}

\foreach \point in {(-1,-3),(-3,-1),(3,1),(1,3)}{\draw[green,fill=green] \point circle (2pt);}
 \node[anchor=south, blue] at (0,-2) {$H^{-2,0}_\infty$};
 \node[anchor=south, gray] at (0,-4) {$H^{-4,0}_\infty$};
 \node[anchor=south, green] at (-1,-3) {$H^{-3,-1}_\infty$};
 \node[anchor=south, gray] at (1,-3) {$H^{-3,1}_\infty$};
 \node[anchor=south, gray] at (-2,-4) {$H^{-4,-2}_\infty$};
 \node[anchor=south, orange] at (2,-4) {$H^{-4,2}_\infty$};

 \node[anchor=south, blue] at (0,2) {$H^{2,0}_\infty$};
 \node[anchor=south, gray] at (0,4) {$H^{4,0}_\infty$};
 \node[anchor=south, green] at (1,3) {$H^{3,1}_\infty$};
 \node[anchor=south, gray] at (-1,3) {$H^{-3,1}_\infty$};
 \node[anchor=south, gray] at (2,4) {$H^{4,2}_\infty$};
 \node[anchor=south, orange] at (-2,4) {$H^{4,-2}_\infty$};

\node[anchor=south, blue] at (-2,0) {$H^{0,-2}_\infty$};
\node[anchor=south, blue] at (2,0) {$H^{0,2}_\infty$};
\node[anchor=south] at (-4,0) {$H^{0,-4}_\infty$};
\node[anchor=south] at (4,0) {$H^{0,4}_\infty$};
\node[anchor=south] at (-3,1) {$H^{1,-3}_\infty$};
\node[anchor=south] at (3,-1) {$H^{-1,3}_\infty$};
\node[anchor=south,green] at (-3,-1) {$H^{-1,-3}_\infty$};
\node[anchor=south, green] at (3,1) {$H^{1,3}_\infty$};
\node[anchor=south] at (4,2) {$H^{2,4}_\infty$};
\node[anchor=south] at (-4,-2) {$H^{-2,-4}_\infty$};
\node[anchor=south,orange] at (4,-2) {$H^{-2,4}_\infty$};
\node[anchor=south,orange] at (-4,2) {$H^{2,-4}_\infty$};

\end{tikzpicture}
\end{center}
\caption{Principal series representations, where the colored one indicate the intertwining relations that occures between each others with the same colors.}
 \label{fig:HSL(2,C)}
\end{figure}
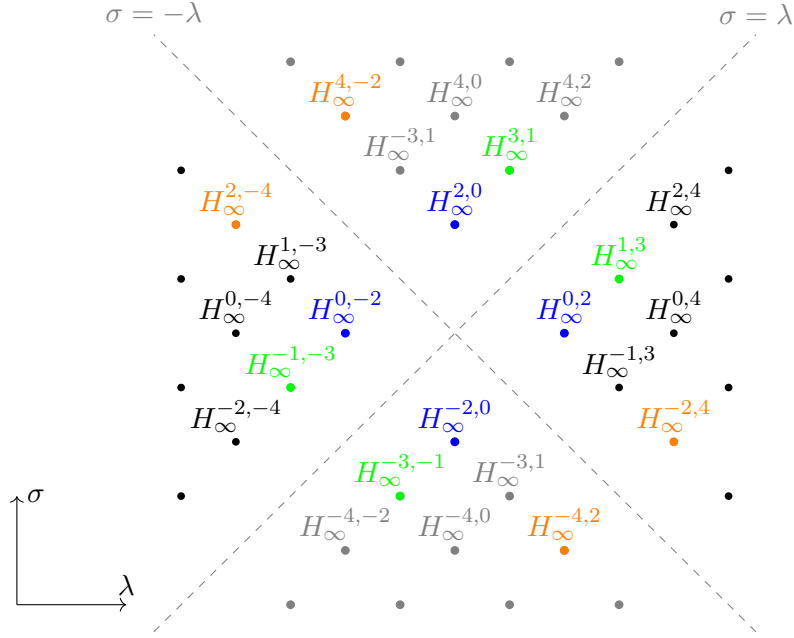

\begin{exmp}
Fix $(\sigma, \lambda) \in \widehat{M} \times \mathfrak{a}^*_\mathbb{C}$ such that $\lambda-\vert\sigma\vert \in 2 \mathbb{N}$, we have the special condition
\begin{equation} \label{eq:Zelointcond}
L_{\sigma,\lambda} \circ \phi(-\lambda,-\sigma) = \phi(\sigma, \lambda) \circ L_{\sigma,\lambda}.
\end{equation}
\end{exmp}

\begin{thm}[Intertwining condition in (\textcolor{blue}{Level 1})] \label{thm:Meta1SL2C}
For $r>0$, let $\mathcal{A}$ be the space of all 
$$\phi \in \prod_{\sigma\in \widehat{M}} \emph \Hol(\mathfrak{a}^*_\mathbb{C},  \emph \End(H^\sigma_\infty))$$
 such that $\phi$ statisfies the corresponding growth condition as well as the two intertwining conditions (\ref{eq:intwcondKS}) and (\ref{eq:Zelointcond}).
Then, $\mathcal{A}$ satisfies the conditions of Thm.~\ref{thm:Meta}, this means that
$\mathcal{A}=PW_r(G).$
\end{thm}

\noindent
Before we proceed with the proof of Thm.~\ref{thm:Meta1SL2C}, let us first state the explicit expression of the Harish-Chandra $\mathbf{c}$-function \cite[App. 2]{Cohn} for $G=\mathbf{SL}(2,\mathbb{C})$, which
is given by the following formula, for $\vert\sigma\vert \leq n, \sigma \equiv n$ (mod 2):
\begin{equation*} \label{eq:cfctSL(2,C)}
\mathbf{c}_{n}(\sigma,\lambda) := \frac{\Gamma(\frac{1}{2}(\lambda+\sigma))\Gamma(\frac{1}{2}(\lambda-\sigma))}{\Gamma(\frac{1}{2}(\lambda+n+2))\Gamma(\frac{1}{2}(\lambda-n))}, \;\;\; \lambda \in\mathfrak{a}^*_\mathbb{C}.
\end{equation*}
Consider an additional, not necessarily distinct, $K$-type $m$ and fix $\lambda \in \mathfrak{a}^*_\mathbb{C}$.
Then, by using repeatedly the relation (\ref{eq:Gammarelation}), we obtain for $n \equiv m$ (mod 2),  the following quotient:
\begin{eqnarray} \label{eq:cfctquotientSL2C}
\frac{\mathbf{c}_{n}(\sigma,\lambda)}{\mathbf{c}_{m}(\sigma,\lambda)}
&=&\frac{\Gamma(\frac{\lambda}{2}+\frac{m}{2}+1)\Gamma(\frac{\lambda}{2}-\frac{m}{2})}{\Gamma(\frac{\lambda}{2}+\frac{n}{2}+1)\Gamma(\frac{\lambda}{2}-\frac{n}{2})}  \nonumber \\
&=&
\begin{cases}
1, &\;\text{if } n=m\\
\frac{(\lambda+m)(\lambda+m-2)(\lambda+m-4) \cdots (\lambda+n+2)}{(\lambda-m)(\lambda-(m+2))(\lambda-(m+4))\cdots (\lambda-(n+2))}, &\;\text{if } n<m\\
\frac{(\lambda-n)(\lambda-(n+2))(\lambda-(n+4))\cdots (\lambda-(m+2))}{(\lambda+n)(\lambda+n-2)(\lambda+n-4) \cdots (\lambda+m+2)}, &\;\text{if } n>m.
\end{cases}
\end{eqnarray}
Hence, we can directly see that the quotient has zeros in $\{-m, -m+2, \dots, -n-~2\}$ and poles in $\{m,m+2, \dots, n+2\}$ for $n < m$ and inversely for $n>m$.

\begin{proof}[Proof of Thm.~\ref{thm:Meta1SL2C}]
$\mathcal{A}$ is a $K\times K$-invariant closed linear subspace.
We proceed similar as in the proof of Thm.~\ref{thm:Meta1}.
Note that $\widehat{G}_d = \emptyset$.
Consider $\phi \in \mathcal{A}$ such that for $\sigma \in \widehat{M}$, the assumption $(D.b) (i)$ of Thm.~\ref{thm:Meta} is satisfied:
\begin{equation} \label{eq:Assump1}
\phi_{\sigma'}=0 \text{ for all } \sigma' \in \widehat{M} \text{ with }\vert\sigma'\vert > \vert\sigma\vert.
\end{equation}
Analogous to the proof of Thm.~\ref{thm:Meta1}, we need to check that:
\begin{itemize}
\item[(a)] for $\text{Re}(\lambda) >0$, the intertwining opertator $J_{w,\sigma,\lambda}$ has a zero of order at most one.
\item[(b)] the kernel of $J_{w,\sigma,\lambda}$ is equal to 0 or $R^{\sigma,\lambda}$ for $\text{Re}(\lambda)>0$.
\end{itemize}
The minimal $K$-type of the principal series representation $H^{\sigma,\lambda}_\infty$ is $n=\vert\sigma\vert$, hence its Harish-Chandra $\mathbf{c}$-function (\ref{eq:cfctSL(2,C)}) $\mathbf{c}_{\sigma,\sigma}$ is regular for $\text{Re}(\lambda) >0$.
Let $n\in \widehat{K}$, then
we see that for $n \geq \vert\sigma\vert$ and for $\text{Re}(\lambda)>0$:
$$\frac{\mathbf{c}_{n}(\sigma,\lambda)}{\mathbf{c}_{\sigma}(\sigma,\lambda)}=\frac{(\lambda-n)(\lambda-(n+2))(\lambda-(n+4))\cdots (\lambda-(\sigma+2))}{(\lambda+n)(\lambda+n-2)(\lambda+n-4) \cdots (\lambda+\sigma+2)} $$
is also regular and has no poles, but zeros 
$\lambda\in \{n,n+2,\dots,\sigma+2\}$ of first order.
Hence due (\ref{eq:normKSintw}), $J_{w,\sigma,\lambda}$ has zeros of order one, this proves the first assertation (a) of the claim.

By Thm.~\ref{thm:WallachSL2C}, we have 
$\text{Ker}(J_{w,\sigma,\lambda})=\text{Im}(L_{\sigma,\lambda})=R^{\sigma,\lambda}$, 
thus, this implies (b).

Now, by putting everyting together and using the intertwining condition (\ref{eq:Zelointcond}) as well as (\ref{eq:Assump1}), we have for each $\text{Re}(\lambda)>0$ with $\vert\lambda\vert-\vert \sigma \vert \in 2 \mathbb{N}$ that
$$\phi_\sigma(\lambda) \circ L_{\sigma,\lambda} = L_{\sigma,\lambda} \circ \underbrace{\phi_{-\lambda}(-\sigma)}_{=0}=0.$$
Thus, by (b) and the assumption, we deduce that the operator $\phi^{(0)}_\sigma(\lambda)$ annihiliates $\text{Ker}(J_{w,\sigma,\lambda})=\text{Im}(L_{\sigma,\lambda})=R^{\sigma,\lambda}$ for  $\vert\lambda\vert > \vert\sigma\vert$ and $\vert\lambda\vert-\vert\sigma\vert$ even.
Moreover, by (a), this condition is sufficient since the order $m$ is one.
This completes the proof.
\end{proof}

Now let us move to (\textcolor{blue}{Level 2}) and state the corresponding intertwining conditions there.
In fact, this will determine explicitly the Paley-Wiener(-Schwartz) spaces for $G=\mathbf{SL}(2, \mathbb{C})$ in (\textcolor{blue}{Level 2}).
In order to distinguish between representation spaces of $K$ and $M$, we denote $E_{\delta_n}$ by $E_n$, while the one-dimensional space $E_{\sigma_l}$ is denoted by $ \mathbb{C}_l$.

\begin{thm}[Intertwining condition in (\textcolor{blue}{Level 2})] \label{thm:Meta2SL2C}
Let $n \in \mathbb{N}_0$ be a $K$-type and $k,l \in \widehat{M} \cong \mathbb{Z}$ so that $n \geq \vert k\vert,\vert l\vert, l > \vert k\vert$ and
$k \equiv l \equiv n$ (mod 2).
Consider the operator
$$l^n_{k,l}: \emph \Hom_M(E_n,  \mathbb{C}_{-l}) \longrightarrow \emph \Hom_M(E_n, \mathbb{C}_{k})$$
defined as in Example~\ref{exmp:IntwOp} corresponding to $L=L_{k,l}$ and $\tau=\delta_n$.
Then, $\psi \in \emph \Hol(\mathfrak{a}^*_ \mathbb{C},H^{n\vert_M}_\infty)$ satisfies the intertwining condition (D.2) of Def.~\ref{defn:intcondL2L3} if, and only if,
\begin{itemize}
\item[(2.a)] 
$J_{w,n,\lambda} \psi(\lambda)=  
\begin{pmatrix} 
c_{n}(n,\lambda) 	&  	0		& \dots	& 0 \\
		0		&c_{n}(n-2,\lambda)	&  \dots 	& 0 \\
			\vdots	&	\vdots			&\ddots	& \vdots \\
	0			&		0		&	\dots	& c_{n}(-n,\lambda)
\end{pmatrix} 
\psi(-\lambda)$
for all  $\lambda \in \mathfrak{a}^*_ \mathbb{C},$
\item[(2.b)] $L_{k, l}( t \circ \psi(-k)) = l^n_{k,l}(t) \circ \psi(l)$ for all  $\lambda \in \mathfrak{a}^*_ \mathbb{C}$ and $t \in \emph \Hom_M(E_n,E_{-l})$.
\end{itemize}
\end{thm}

\begin{proof}[Proof of Thm.~\ref{thm:Meta2SL2C}]
By Prop.~\ref{prop:normKSLevel2}, we have that $(2.a)$ corresponds to (\ref{eq:Lintwcond2}), hence it corresponds to the intertwining condition $(D.2)$ in Def.~\ref{defn:intcondL2L3}.
Similar for $(2.b)$, which is, by Example~\ref{exmp:IntwOp} (\ref{eq:Lintwcond2}), a special intertwining condition of $(D.2)$.
Hence, we have equivalence between the conditions $(2.a)$, $(2.b)$ and $(D.2)$.
\end{proof}

In order to move to (\textcolor{blue}{Level 3}), let us first consider the case where the two $K$-types $(n, E_n)$ and $(m, E_m)$ are equal and then progress to the case of distinct $K$-types.

\subsection*{Initial case: The $K$-types $m$ and $n$ are equal}

\begin{defn} \label{def:Amm}
Let $(m,E_m)$ be a fixed $K$-type and $k=-m,-(m-2), \dots, m-2,m$.\\
We define $\prescript{}{m}{\mathcal{A}}_{m}$ as the space of all elements in $ \emph \Hol(\mathfrak{a}^*_ \mathbb{C}, \emph \End_M(E_m))$, which are given by holomorphic functions $\varphi_k:\mathfrak{a}^*_ \mathbb{C} \rightarrow \mathbb{C}, $ ordered to a $(m+1) \times (m+1)$ diagonal matrix with respect to $M$-weight vectors as a basis of $E_m$:
$$
\varphi:=\begin{pmatrix} 
\varphi_{m}(\lambda) 	& 0 					& \cdots 					& 0 \\
0 				& \varphi_{m-2}(\lambda)	& \cdots 					& 0 \\
\vdots 			& \vdots 				& \ddots 					&\vdots\\
0 				&  0  				& \cdots						& \varphi_{-m}(\lambda) \\
\end{pmatrix}\in \emph \Hol(\mathfrak{a}^*_ \mathbb{C},  \emph \End_M(E_m))$$
such that
\begin{eqnarray} \label{eq:propSL2C}
\varphi_k(\lambda)&=&\varphi_{-k}(-\lambda),\; \text{ for } \lambda \in \mathfrak{a}^*_ \mathbb{C} \nonumber \\
\varphi_{k}(l)&=& \varphi_{l}(k), \; \text{ for } k \equiv l \equiv m \text{ (mod } 2) \;\text{and}\; \vert k\vert ,\vert l\vert \leq m.
\end{eqnarray}
\end{defn}

\noindent
Note that $\prescript{}{m}{\mathcal{A}}_{m} \subset \text{Hol}(\mathfrak{a}^*_ \mathbb{C},   \text{End}_M(E_m))$ is an algebra.
Let $\prescript{}{m}{PWS}_{m,H}(\mathfrak{a}^*_ \mathbb{C})$ be the pre-Paley-Wiener-Schwartz space consisting of all $\varphi \in \text{Hol}(\mathfrak{a}^*_ \mathbb{C}, \text{End}_M(E_m))$ satisfying the intertwining conditions $(D.3)$ of Def.~\ref{defn:intcondL2L3}.

\begin{thm} \label{thm:IsomAm}
With the previous notations, we have that
$$\prescript{}{m}{\mathcal{A}}_{m} \cong \prescript{}{m}{PWS}_{m,H}(\mathfrak{a}^*_ \mathbb{C}).$$
\end{thm}

\begin{proof}
We need to check that the intertwining conditions (\ref{eq:propSL2C}) of $\prescript{}{m}{\mathcal{A}}_{m}$ correspond to the intertwining condition $(D.3)$ of $\prescript{}{m}{PWS}_{m,H}(\mathfrak{a}^*_ \mathbb{C})$.
More precisely, by Example~\ref{exmp:IntwOp}, it suffices to show that the intertwining condition (\ref{eq:Lintwcond3}) for $L=J_{w,\sigma,\lambda}$ and $L=L'_{\sigma,\lambda}$ corresponds to (\ref{eq:propSL2C}).

We know from Prop.~\ref{prop:normKSLevel2} with $m=\gamma=\tau \in \mathbb{N}_0$ that
$$\delta_m(m_w)^{-1}\varphi(\lambda) \delta_m(m_w)=\varphi(-\lambda), \;\; \lambda \in \mathfrak{a}^*_ \mathbb{C}, \varphi \in \text{Hol}(\mathfrak{a}^*_ \mathbb{C}, \text{End}_M(E_m)).$$
Note that the complex hull of $\mathfrak{a}$ is the sum of $\mathfrak{a}$ and $\mathfrak{m}$.
This means that the Weyl group $W_A$ acts on $\mathfrak{a}$ as well as on $\mathfrak{m}$ by $-1$ (thus also on $i\mathfrak{m}^*$ by $-1$).
Therefore the matrix $\text{diag}(\varphi_m(\lambda), \dots, \varphi_{-m}(\lambda))$ is reversed by conjugation by $\delta_m(m_w)$, i.e., we get $\text{diag}(\varphi_{-m}(-\lambda), \dots, \varphi_{m}(-\lambda))$.
Hence, $\varphi_{-k}(-\lambda)=\varphi_k(\lambda)$, for all $\vert k\vert  \leq m$ and $\lambda \in \mathfrak{a}^*_ \mathbb{C}$.

Let $\vert l\vert,\vert k\vert \leq m$ and $k \equiv l \equiv m$ (mod 2).
Let $l^m_{k,l}$ as in Thm.~\ref{thm:Meta2SL2C}.
We know that for $t\in \text{Hom}_M(E_m,\mathbb{C}_{-l})$ we have by (\ref{eq:Lintwcond3}) for $m=\gamma=\tau \in \mathbb{N}_0$ and $l > \vert k\vert$
$$l^m_{k,l}(t \circ \varphi(-k))=l^m_{k,l}(t) \circ \varphi(l).$$
Let $t\in \text{Hom}_M(E_m,\mathbb{C}_{-l})$ and $t' \in \text{Hom}_M(E_m,\mathbb{C}_k)$ be such that $\varphi_{-l}=t\circ \varphi$ and $\varphi_k=t' \circ \varphi$.
We have that 
$l^m_{k,l}(t)=c \cdot t'$, where $c\in \mathbb{C}$.
Note that $c$ is non-zero. In fact, by Thm.~\ref{thm:WallachSL2C}, the intertwining operator $L_{k,l}$ on the $K$-type $m$ is not zero, hence $l^m_{k,l}$ too.
Consequently, we have
\begin{eqnarray} \label{eq:kl}
l^m_{k,l}(t \circ \varphi(-k))=l^m_{k,l}(t) \circ \varphi(l)
\iff c\cdot \varphi_{-l}(-k)= c\cdot \varphi_k(l) 
&\iff& \varphi_{-l}(-k)=\varphi_k(l). \nonumber \\
&&
\end{eqnarray}
For $0 \leq l < \vert k \vert$ we consider $L_{l,\vert k\vert}$ instead. Combined with $\varphi_{-l}(-k)=\varphi_l(k)$, we obtain (\ref{eq:kl}) for every pair $k,l$ as above.
This completes the proof.
\end{proof}

We also consider the corresponding situation in polynomial functions:
$$\prescript{}{m}{\text{Pol}}_{m}:=\{\varphi \in \text{Pol}(\mathfrak{a}^*_\mathbb{C},\text{End}_M(E_m)) \;\vert\; \varphi \text{ satisfies } (\ref{eq:propSL2C})\}.$$ 
It follows from Thm.~\ref{thm:IsomAm} that $\prescript{}{m}{\text{Pol}}_{m}$ is equal to the vector space $\prescript{}{m}{PWS}_{m,0}(\mathfrak{a}^*_\mathbb{C})$.
We will sometimes write elements of $\prescript{}{m}{\text{Pol}}_{m}$ and $\prescript{}{m}{\mathcal{A}}_{m}$ as functions of two variables: 
$$\varphi(\lambda,k):=\varphi_k(\lambda).$$
We consider the subalgebra of $\prescript{}{m}{\text{Pol}}_{m}$ generated by $\lambda^2+k^2$.
It is the subalgebra generated by the Fourier image of the Casimir operator and isomorphic to $\text{Pol}(\mathbb{C})$.
Thus, we can view $\prescript{}{m}{\text{Pol}}_{m}$ as a $\text{Pol}(\mathbb{C})$-module.
Similarly, $\prescript{}{m}{\mathcal{A}}_{m}$ has the structure of $\text{Hol}(\mathbb{C})$-module.
Here, $h \in \text{Hol}(\mathbb{C})$ acts on $\prescript{}{m}{\mathcal{A}}_{m}$ by
$$(h\cdot \varphi)_k(\lambda)=h(\lambda^2+k^2) \varphi_k(\lambda), \;\; h\in \text{Hol}(\mathbb{C}), \varphi \in \prescript{}{m}{\mathcal{A}}_{m}.$$
\begin{thm} \label{prop:Amm}
The algebra $\prescript{}{m}{\mathcal{A}}_{m}$ is a free $\emph \Hol(\mathbb{C})$-module with the $m+1$ generators  $(k \lambda)^l \in \prescript{}{m}{\emph \Pol}_{m} \subset \prescript{}{m}{\mathcal{A}}_{m}, l=0, \dots, m$.
Furthermore, we have
\begin{equation} \label{eq:Ammflat}
\prescript{}{m}{\mathcal{A}}_{m} \cong \emph \Hol(\mathbb{C}) \otimes_{\emph \Pol(\mathbb{C})} \prescript{}{m}{\emph \Pol}_{m}.
\end{equation}
Analogously, $\prescript{}{m}{\emph \Pol}_{m}$ is a free $\emph \Pol(\mathbb{C})$-module with same generators as $\prescript{}{m}{\mathcal{A}}_{m}$.
\end{thm}

\noindent
Note that Thm.~\ref{prop:Amm} also tells that the two elements $\lambda^2+k^2$ and $kl$ generate $\prescript{}{m}{\text{Pol}}_{m}$ as an algebra.
Observe also that $\prescript{}{m}{\text{Pol}}_{m}$ is isomorphic to $\mathcal{D}_G(\mathbb{E}_m,\mathbb{E}_{m})$, the set of all invariant differential operators $D : C^\infty(X,\mathbb{E}_m) \longrightarrow C^\infty(X, \mathbb{E}_{m})$ \cite[Sect.~7]{PalmirottaPWS}.

\begin{proof}
Consider $\varphi \in \prescript{}{m}{\mathcal{A}}_{m}$.
It is sufficient to show the existence and the uniqueness of holomorphic functions $h_0, \dots, h_m \in \text{Hol}(\mathbb{C})$ so that
\begin{equation} \label{eq:phik}
\varphi_k(\lambda):=\sum_{l=0}^m h_l(\lambda^2+k^2)\cdot(k \lambda)^l, \;\;\; \text{ for } k=-m, \dots,m
\end{equation}
and 
\begin{equation} \label{eq:phik2}
\varphi \in \prescript{}{m}{\text{Pol}}_{m} \text{ implies } h_l \in \text{Pol}(\mathbb{C}), \;\;\; \text{ for } l=0, \dots,m.
\end{equation}
Then, $\prescript{}{m}{\mathcal{A}}_{m}$ is a free $\text{Hol}(\mathbb{C})$-module with generators $(k\lambda)^l, l=0,\dots,m$. Similarly for $\prescript{}{m}{\text{Pol}}_{m}$.
Note that $\text{Hol} \otimes_{\text{Pol}} \text{Pol} \cong \text{Hol}$. Since
there are $m+1$ free generators we have
$$\prescript{}{m}{\mathcal{A}}_{m} \cong \text{Hol}(\mathbb{C})^{m+1} \text{ and }
\prescript{}{m}{\text{Pol}}_{m} \cong \text{Pol}(\mathbb{C})^{m+1}.$$
Thus also (\ref{eq:Ammflat}) follows.

For the existence of $h_l \in \text{Hol}(\mathbb{C}),$ $l=0, \dots, m$,
we proceed by a step two induction on $m \in \mathbb{N}_0$.\\
For $m=0$, we have
$$\prescript{}{0}{\mathcal{A}}_{0}=\{ \varphi_0\in \text{Hol}(\mathfrak{a}^*_\mathbb{C}) \;\vert\; \varphi_0(\lambda)=\varphi_0(-\lambda), \forall \lambda \in \mathfrak{a}^*_\mathbb{C}\}.$$
We see immediately that there is exactly one holomorphic function $h_0$ such that $\varphi_0(\lambda)=h_0(\lambda^2)$.
$\varphi_0$ is a polynomial function if and only if $h_0$ is one.\\
For $m=1$, we have that
$\prescript{}{1}{\mathcal{A}}_{1}=\{(\varphi_1, \varphi_{-1}) \in  \text{Hol}(\mathfrak{a}^*_\mathbb{C})^2 \;\vert\; \varphi_{-1}(\lambda)=\varphi_1(-\lambda) \forall \lambda \in \mathfrak{a}^*_\mathbb{C}\}.$
$\varphi_1$ can be decomposed into an even and an odd part as follows:
$$\varphi_1(\lambda)=\varphi^{\text{even}}_1(\lambda)+\varphi^{\text{odd}}_1(\lambda)
= h_0(\lambda^2+1) + \lambda h_1(\lambda^2+1).$$
Then, $\varphi_{-1}(\lambda)=\varphi_1(-\lambda)=h_0(\lambda^2+1) - \lambda h_1(\lambda^2+1).$
Hence, this leads us to desired relation $\varphi_k(\lambda)=h_0(\lambda^2+k^2)+(k \lambda)h_1(\lambda^2+k^2)$ for $k = \pm 1$.
$h_0$ and $h_1$ are polynomials if and only if $\varphi_{\pm 1}$ are polynomials as well.\\
Assume now the existence of $h_l$ satisfying (\ref{eq:phik}) and (\ref{eq:phik2}) for $m$ replaced by $m-2$. 
Let $\varphi \in \prescript{}{m}{\mathcal{A}}_{m}$.
We have
$$\overline{\varphi}:=\text{diag}(\varphi_{m-2},\varphi_{m-4} , \dots, \varphi_{-(m-4)}, \varphi_{-(m-2)})
\in \prescript{}{m-2}{\mathcal{A}}_{m-2}$$
so that, by induction hypothesis, there exists $h_l \in \text{Hol}(\mathbb{C})$ with 
$$\varphi_k(\lambda)=\sum_{l=0}^{m-2} h_l(\lambda^2+k^2)(k\lambda)^l \;\;\; \text{ for } \vert k \vert \leq m-2,$$
and $h_l \in \text{Pol}(\mathbb{C})$, if $\varphi \in \prescript{}{m}{\text{Pol}}_m$.
Consider 
$$\tilde{\varphi}:=\text{diag}(\tilde{\varphi}_{m}, \varphi_{m-2}, \dots, \varphi_{-(m-2)}, \tilde{\varphi}_{-m})
\in \prescript{}{m}{\mathcal{A}}_{m}$$
with $$\tilde{\varphi}_{\pm m}(\lambda):=\sum_{l=0}^{m-2} h_l(\lambda^2+m^2)(\pm m\lambda)^l.$$
By taking the difference of $\varphi$ and $\tilde{\varphi}$, we get that 
$$\varphi-\tilde{\varphi}=\text{diag}(\varphi^+_{m}, 0, \dots, 0, \varphi^+_{-m})
\in \prescript{}{m}{\mathcal{A}}_{m},$$
where we have set $\varphi^+_{\pm m}(\lambda):= \varphi_{\pm m}(\lambda)-\tilde{\varphi}_{\pm m}(\lambda)$.
Notice that $\varphi^+_{\pm m}(l)=0$ for $\vert l\vert \leq m-2, l \equiv m$ (mod 2).
We introduce the polynomial function
$$
p_m(\lambda,k):= \prod_{\substack{ \vert l\vert \leq m-2\\ l \equiv m \text{ (mod 2)}}} (k-l)(\lambda-l) \in \prescript{}{m}{\text{Pol}}_{m}.
$$
Note that $p_m(\lambda,k) \equiv 0$ for $\vert k\vert \leq m-2, k \equiv m$ (mod 2).
Moreover, if $k=m$, then 
$p_m(\lambda,m)=c_m\prod_{\substack{ \vert l\vert  \leq m-2\\ l \equiv m \text{ (mod 2)}}} (\lambda-l),$
where $c_m$ is a non-zero constant depending on the integer $m$.
We conclude that there exist $h^+_0, h^+_1 \in \text{Pol}(\mathbb{C})$, if $\varphi \in \prescript{}{m}{\text{Pol}}_m$ such that
$$\varphi^+_m(\lambda)= [h^+_0(\lambda^2+m^2)+h^+_1(\lambda^2+m^2)(m\lambda)] p_m(\lambda,m).$$
This implies that $(\varphi-\tilde{\varphi})(\lambda,k)= [h^+_0(\lambda^2+k^2)+ h^+_1(\lambda^2+k^2)(k\lambda)] p_m(\lambda,k)$.
In addition,
$(k-\lambda)(\lambda-l)(k+l)(\lambda+l)=(k^2-l^2)(\lambda^2-l^2)=(k\lambda)^2-l^2(\lambda^2+k^2)+l^4$
and thus $p_m(\lambda,k)$ is of the form:
\begin{equation} \label{eq:pmk} 
p_m(\lambda,k)=(k\lambda)^{m-1} + \sum_{\substack{l=0 \\ l \equiv m-1 \text{ (mod 2)}}}^{m-3} p^m_l(\lambda^2+k^2)(k\lambda)^l,
\end{equation}
where $p^m_l$ are certain polynomials.
We obtain
\begin{eqnarray*} 
(\varphi-\tilde{\varphi})(\lambda,k)
&{=}& h^+_0(\lambda^2+k^2)(k\lambda)^{m-1}+ h^+_0(\lambda^2+k^2)\sum_{\substack{l=0 \\ l \equiv m-1 \text{ (mod 2)}}}^{m-3} p^m_l(\lambda^2+k^2)(k\lambda)^l\\
&&+h^+_1(\lambda^2+k^2)(k\lambda)^m +h^+_1(\lambda^2+k^2)\sum_{\substack{l=0 \\ l \equiv m-1 \text{ (mod 2)}}}^{m-3} p^m_l(\lambda^2+k^2)(k\lambda)^{l+1}.
\end{eqnarray*}
This implies that $\varphi-\tilde{\varphi}$ is of the desired form, hence $\varphi$.

Concerning the uniqueness, we need to show that 
\begin{equation} \label{eq:phik0}
\sum_{l=0}^m h_l(\lambda^2+k^2)(k\lambda)^l=0, \; \forall \vert k\vert\leq m, \lambda \in \mathfrak{a}^*_\mathbb{C}, \text{ implies } h_l=0,  l=0, \dots,m.
\end{equation}
We proceed again by a two step induction on $m \in \mathbb{N}_0$.
For the initial cases $m=0,1$, the assertation is clear, see above.
Assume that (\ref{eq:phik0}) holds true for $m-2$ and let us prove that it holds for $m$. 
Let $\varphi \in \prescript{}{m}{\mathcal{A}}_m$.
By using the polynomial function (\ref{eq:pmk}), we have for $\vert k\vert \leq m-2$:
\begin{eqnarray*}
0
&=& \sum_{l=0}^m  h_l(\lambda^2+k^2)(k\lambda)^l  - [h_{m-1}(\lambda^2+k^2) +(k\lambda) h_m(\lambda^2+k^2)p_m(\lambda,k)] \\
&=&\sum_{l=0}^{m-2} h_l(\lambda^2+k^2)(k\lambda)^l
- h_{m-1}(\lambda^2+k^2) \sum_{\substack{l=0 \\ l \equiv m-1 \text{ (mod 2)}}}^{m-3} p^m_l(\lambda^2+k^2)(k\lambda)^l \\
&&- h_{m}(\lambda^2+k^2) \sum_{\substack{l=0 \\ l \equiv m-1 \text{ (mod 2)}}}^{m-3} p^m_l(\lambda^2+k^2)(k\lambda)^{l+1} \\
\end{eqnarray*}
in $\prescript{}{m-2}{\mathcal{A}}_{m-2}$.
By induction hypothesis, this implies that for $l\leq m-2$:
$$h_l(\mu)=
\begin{cases}
h_{m-1}(\mu)p^m_{l}(\mu), & l \equiv m-1 \text{ (mod 2)} \\
h_m(\mu) p^m_{l-1}(\mu),& l \equiv m \text{ (mod 2)}.
\end{cases}
$$
Then, for $k=m$, this implies
$$0=\sum_{l=0}^m h_l(\lambda^2+m^2)(k\lambda)^l=[h_{m-1}(\lambda^2+m^2)+h_m(\lambda^2+m^2)(m\lambda)]p_m(\lambda, m).$$
Since $p_m(\lambda,m)$ is not identical zero on $\lambda \in \mathfrak{a}^*_\mathbb{C}$, we
 obtain that $h_{m-1}$ and $h_m$ are zero. Hence $h_l=0$, for $l\leq m$.
This completes the proof.
\end{proof}

\subsection*{General case: The $K$-types $n$ and $m$ are distinct}
Consider now two distinct $K$-types $(n,E_n)$ and $(m,E_m)$.
Since $\text{Hom}_M(E_n,E_m)=0$, for $n \not\equiv m$ (mod 2), we assume throughout the section that $n \equiv m$ (mod 2).
We define 
$$\prescript{}{n}{\mathcal{A}}_{m} := \prescript{}{n}{PWS}_{m,H}(\mathfrak{a}^*_\mathbb{C})\subset \text{Hol}(\mathfrak{a}^*_\mathbb{C}, \text{Hom}_M(E_n,E_m))$$
and
$\prescript{}{n}{\text{Pol}}_{m} := \prescript{}{n}{PWS}_{m,0}(\mathfrak{a}^*_\mathbb{C}).$
Recall that $\prescript{}{n}{PWS}_{m,H}(\mathfrak{a}^*_\mathbb{C})$ denotes the pre-Paley-Wiener-Schwartz space and consists of all $\varphi \in \text{Hol}(\mathfrak{a}^*_\mathbb{C}, \text{Hom}_M(E_n,E_m))$ satisfying the intertwining condition $(D.3)$ of Def.~\ref{defn:intcondL2L3}.
Moreover, 
$$\prescript{}{n}{PWS}_{m,0}(\mathfrak{a}^*_\mathbb{C})=\prescript{}{n}{PWS}_{m,H}(\mathfrak{a}^*_\mathbb{C}) \cap
\text{Pol}(\mathfrak{a}^*_\mathbb{C}, \text{Hom}_M(E_n,E_m)).$$
Note that after a choice of bases in $E_n,E_m$ consisting of $M$-weight vectors, we can write elements of $\text{Hol}(\mathfrak{a}^*_\mathbb{C}, \text{Hom}_M(E_n,E_m))$ in the following way:
$$
\varphi:=
\begin{cases}
\begin{pmatrix}
0						& \cdots 	& 0 \\
\vdots 					& \vdots 	&\vdots\\
0						& \cdots 	& 0 \\
\varphi_n(\lambda) 			& \cdots 	& 0 \\
\vdots 					& \ddots 	&\vdots\\
0 				 		& \cdots 	& \varphi_{-n}(\lambda) \\
0						& \cdots 	& 0\\
\vdots 					& \vdots 	&\vdots\\
0						& \cdots 	& 0 
\end{pmatrix}_{(m+1) \times (n+1)}& \text{ if } n < m,\\
\\
\begin{pmatrix}
0 		& \cdots 	& 0 		&\varphi_m(\lambda) 	& \cdots 	& 0 				&0 			& \cdots   	& 0\\
\vdots 	& \vdots 	& \vdots  	& \vdots			& \ddots 	& \vdots  			&  \vdots 		& \vdots  	& \vdots\\
0		& \cdots 	& 0		& 0	  			& \cdots 		& \varphi_{-m}(\lambda)	& 0 			& \cdots 	& 0 \\
\end{pmatrix}_{(m+1) \times (n+1)}&  \text{ if } n > m.
\end{cases}
$$
\noindent
Note that $\prescript{}{n}{\text{Pol}}_m$ is isomorphic to the set of all $G$-invariant differential operators $\mathcal{D}_G(\mathbb{E}_n,\mathbb{E}_m)$  \cite[Sect.~7]{PalmirottaPWS}.
We start with the case $\vert n-m\vert =2$.

\begin{prop} \label{prop:q+m}
Let $m$ be a $K$-type.
There exists a unique operator of first order $q^+_m$  in $\prescript{}{m}{\emph \Pol}_{m+2}$ (resp. $q^-_m$ in $\prescript{}{m+2}{\emph \Pol}_{m}$), up to normalization. It corresponds to
\begin{eqnarray*}
&(\lambda+m+2) \cdot &
\begin{pmatrix}
0								& \cdots 	& 0 \\
1 					& \cdots 	& 0 \\
\vdots 			& \ddots 					&\vdots\\
0 				&  \hdots  						& 1\\
0							& \cdots 	& 0
\end{pmatrix}_{(m+2) \times m} \nonumber \\
&\Bigg(\text{resp. } (\lambda-(m+2)) \cdot &
\begin{pmatrix}
0 		&d(m,m) 				& \cdots 	&0 & 0\\
\vdots 	& \vdots 			& \ddots 	&\vdots & \vdots\\
0 		&  0 	 		& \hdots		&d(m,-m)& 0 \\
\end{pmatrix}_{m \times (m+2)}\Bigg)
 \end{eqnarray*}
under some appropriate basis choice.
Here $d(m,k)=(m+2)^2-k^2$.
\end{prop}

\begin{proof}
Consider the symmetric algebra $S(\mathfrak{p})$ of $\mathfrak{p}\cong \mathbb{R}^3$ over $\mathbb{C}$.
It follows from the Poincar\'e-Birkhoff-Witt theorem and the description of invariant differential opeartors in terms of $\mathcal{U}(\mathfrak{g})$ that the associated graded module of $\prescript{}{m}{\text{Pol}}_{m+2}$ is isomorphic to $[S(\mathfrak{p}) \otimes \text{Hom}_M(E_m,E_{m+2})]^K$ (e.g. \cite[Folgerung 2.5]{Martin} or \cite{PalmirottaDiss}).
Using the Clebsch-Gordan rule (\ref{eq:CGrule}) we obtain
\begin{eqnarray*}
[S^{\leq 1}(\mathfrak{p}) \otimes \text{Hom}(E_m,E_{m+2})]^K 
&\cong& [(S^0(\mathfrak{p}) \oplus S^1(\mathfrak{p})) \otimes (E_2 \oplus E_4 \oplus \dots \oplus E_{2m+2})]^K \\
&\cong& [S^1(\mathfrak{p}) \otimes E_2]^K,
\end{eqnarray*}
where $[S^1(\mathfrak{p}) \otimes E_2]^K$ is one-dimensional.
Note that $S^1(\mathfrak{p}) \cong \mathfrak{p}_\mathbb{C} \cong E_2$.
This means that $\prescript{}{m}{\text{Pol}}_{m+2}$ contains exactly one element $q^+_m$ of filter degree one (up to normalization).
After an appropiate choice of basis, we can write $q^+_{m}$ in matrix form as above.
Since $q^+_{m}$
is of first order (but note that the individual components could be constants), we can find basis element such that
$$q^+_{m}(\lambda)=
\begin{cases}
\tilde{c}(m,k),&   \text{ or } \\
\lambda + c(m,k),&  \text{ (at least for one $k$)}.
\end{cases}$$
Here ${c}(m,k), \tilde{c}(m,k)$ are constants depending on $k$ and $m$. 
Consider $\varphi_k \in \prescript{}{m}{\mathcal{A}}_{m+2}$.
By Prop.~\ref{prop:normKSLevel2}~(b) and formula (\ref{eq:cfctquotientSL2C}), we have that
\begin{eqnarray*}
\varphi_{-k}(-\lambda)
= (-1)^{(m+2-m)/2} \frac{\mathbf{c}_{m+2}(k,\lambda)}{\mathbf{c}_{m}(k,\lambda)} \varphi_k(\lambda) 
&=&(-1) \frac{\Gamma(\frac{\lambda}{2}+\frac{m}{2}+1)\Gamma(\frac{\lambda}{2}-\frac{m}{2})}{\Gamma(\frac{\lambda}{2}+\frac{m}{2}+2)\Gamma(\frac{\lambda}{2}-\frac{m}{2}-1)} \varphi_k(\lambda) \\
&=&(-1)\frac{(\lambda-(m+2))}{(\lambda+(m+2))}\varphi_k(\lambda) \\
&=& \frac{(-\lambda+(m+2))}{(\lambda+(m+2))} \varphi_k(\lambda).
\end{eqnarray*}
Since $\varphi_{-k}$ has no singularities, we conclude that $\varphi_k$ has a zero at $-(m+2)$.
Applying this to $\varphi=q^+_m$, we see that
$$q^+_{m,k}(\lambda)= \begin{cases} 0, \\ \lambda +(m+2). \end{cases}$$
Analogously, there exists exactly one (up to normalization) $q^-_m \in \prescript{}{m+2}{\mathcal{A}}_m$ of degree 1.
We consider $q^+_mq^-_m \in \prescript{}{m+2}{\mathcal{A}}_{m+2}$. Then $q^+_mq^-_m(\lambda, \pm (m+2))=0$.
This implies that $q^+_mq^-_m(\pm(m+2),k)=0$ for $\vert k\vert \leq m, k \equiv m$ (mod 2) of degree 2.
Hence, $q^+_mq^-_m(\lambda,k)=d(m,k)(\lambda^2-(m+2)^2)$ for some constants $d(m,k)$.
Since $q^-_m=c \cdot (q^+_m)^*$, where $*$ stands for the adjoint $p^*(\lambda):=p(-\overline{\lambda})^*$, we have $d(m,k) \neq 0$ for at least one $k$.
In addition, by (\ref{eq:propSL2C})
$$
d(m,k)= \begin{cases}
d(m,0) \cdot \frac{(m+2)^2-k^2}{(m+2)^2},& m \text{ even}\\
d(m,1) \cdot \frac{(m+2)^2-k^2}{(m+2)^2-1},& m \text{ odd}.
\end{cases}
$$
We conclude (up to normalization) that $q^+_m(\lambda,k)= \lambda+c(m,k)=\lambda+(m+2)^2$
and $$q^-_m(\lambda,k)=((m+2)^2-k^2)(\lambda-(m+2)).$$
\end{proof}

Now we consider general $n,m \in \mathbb{N}_0, n \equiv m$ (mod 2).
\begin{defn} \label{defn:qnmSL2C}
Consider $q^\pm_m$ and $q^\pm_n$ as in Prop.~\ref{prop:q+m}.
We define the polynomial $q_{n,m} \in \prescript{}{n}{\text{Pol}}_m$ by
\begin{equation} \label{eq:qnmSL2C}
q_{n,m}=
\begin{cases}
q^+_{m-2} \cdot q^+_{m-4} \cdots q^+_{n+2} \cdot q^+_n, & \text{if } n<m\\
q^-_{m} \cdot q^-_{m+2} \cdots q^-_{n-4} \cdot q^-_{n-2}, & \text{if } n>m\\
\text{Id},& \text{if } n=m.
\end{cases}
\end{equation}
\end{defn}

Finally, the following theorem gives explicitly the Paley-Wiener(-Schwartz) spaces in (\textcolor{blue}{Level 3}) for $G=\mathbf{SL}(2,\mathbb{C})$.

\begin{thm}[Intertwining condition in (\textcolor{blue}{Level 3})]\label{thm:Meta3SL2C}
Let $n,m \in \mathbb{N}_0$ be two $K$-types, which are not necessarily distinct, and let $l:=\min(n,m)$.
Then $\prescript{}{n}{\emph \Pol}_{m}$ (resp. $\prescript{}{n}{\mathcal{A}}_{m}$) is a free $\prescript{}{l}{\emph \Pol}_{l}$ (resp. $\prescript{}{l}{\mathcal{A}}_{l}$)-module with generator $q_{n,m}$.
This means that there exists a unique function $h \in \prescript{}{l}{\mathcal{A}}_l$ such that
\begin{equation} \label{eq:intwcondSL2C}
\prescript{}{n}{\mathcal{A}}_{m} \ni \varphi(\lambda)=
\begin{cases}
h(\lambda) q_{n,m}(\lambda), & \text{ if } m < n,\\
q_{n,m}(\lambda) h(\lambda), & \text{ if } m > n,\\
\end{cases}
\;\;\;\; \text{for } \lambda \in \mathfrak{a}^*_\mathbb{C}.
\end{equation}
Moreover, if $L=\max(n,m)$, then $\prescript{}{n}{\emph \Pol}_{m}$ (resp. $\prescript{}{n}{\mathcal{A}}_{m}$) is a $\prescript{}{L}{\emph \Pol}_{L}$ (resp. $\prescript{}{L}{\mathcal{A}}_{L}$)-module generated by $q_{n,m}$.
\end{thm}

\begin{proof}
Consider the case $m<n$, it suffices to prove that
\begin{itemize}
\item[(a)] there exists a unique $h\in \prescript{}{m}{\mathcal{A}}_{m}$ such that $\varphi=h\cdot q_{n,m} \in \prescript{}{n}{\mathcal{A}}_{m}$,.
\item[(b)]  there exists a $\tilde{h} \in \prescript{}{n}{\mathcal{A}}_{n}$ such that $\varphi=q_{n,m}\cdot \tilde{h} \in \prescript{}{n}{\mathcal{A}}_{m}$.
\end{itemize}
The polynomial case as well as $m>n$ can be proved in a similar way.
Consider
$$q_{m,n} \prescript{}{n}{\mathcal{A}}_{m} \subset \{g \in \prescript{}{n}{\mathcal{A}}_{n} \;\vert\; g_k=0,\; \forall \vert k\vert > m\} \subset \prescript{}{n}{\mathcal{A}}_{n}.$$
We also have zeros between the lines $-m$ and $m$, this means $g_l(k)=0, k = \pm (m+2), \pm (m+4), \dots, \pm n$.
Let $\varphi \in \prescript{}{n}{\mathcal{A}}_{m}$.
Every component of $g=q_{m,n} \cdot \varphi$ has zeros as described above and $q_{m,n}$ has zeros only at negative $\lambda$. It follows that every component $\varphi$ has zeros at $m+2, \dots, n$.
Hence, there exists a unique $h \in \text{Hol}(\mathfrak{a}^*_\mathbb{C},\text{End}(E_m))$ so that 
$$\varphi=h \cdot q_{n,m}.$$
Since $g$ satisfies the intertwining conditions $g_k(l)=g_l(k), g_k(\lambda)=g_{-k}(-\lambda)$, and 
\begin{eqnarray*}
q_{m,n}(k,\lambda)q_{n,m}(k,\lambda)&=&q_{m,n}(-k,-\lambda)q_{n,m}(-k,-\lambda) \\
q_{m,n}(k,l)q_{n,m}(k,l)&=&q_{m,n}(l,k)q_{n,m}(l,k)
\end{eqnarray*}
we see that $h_l(k)=h_k(l)$ and $h_k(\lambda)=h_{-k}(-\lambda)$
for $k\equiv l \equiv m$ (mod 2), $\vert k \vert \leq m$ and $\lambda \in \mathfrak{a}^*_\mathbb{C}$.
This proves (a).

Concerning (b), we know from (a) that $\varphi =h\cdot q_{n,m}$.
Thus we need to find 
$$\tilde{h}=\text{diag}(\tilde{h}_n,\dots, \tilde{h}_{m+2}, \tilde{h}_m, \dots, \tilde{h}_{-m}, \tilde{h}_{-(m-2)}, \dots, \tilde{h}_{-n})
\in \prescript{}{n}{\mathcal{A}}_{n}
$$ 
with $\tilde{h}_k=h_k$ for $\vert k \vert \leq m$.
Note that the crucial condition to be satisfied is $\tilde{h}_k(l)=\tilde{h}_l(k).$
Then $\varphi= \prescript{}{n}{q}_m \cdot \tilde{h}$.
By using interpolation polynomials, we define recursively
\begin{eqnarray*}
\tilde{h}_{m+2}(\lambda)&:=&\sum_{\substack{i=-m \\ i \equiv m \text{ (mod 2)}}}^m h_i(m+2)\prod^m_{\substack{ l =-m \\ l \neq i}} \Big( \frac{\lambda-l}{i-l}\Big) \\
\tilde{h}_{m+4}(\lambda)&:=&\sum_{\substack{i=-(m+2) \\ i \equiv m+2 \text{ (mod 2)}}}^{m+2} \tilde{h}_i(m+4)\prod^{m+2}_{\substack{ l =-(m+2) \\ l \neq i}} \Big( \frac{\lambda-l}{i-l}\Big)
\end{eqnarray*}
\begin{eqnarray*}
&\vdots& \\
\tilde{h}_{n-2}(\lambda)&:=&\sum_{\substack{i=-(n-4) \\ i \equiv n-4 \text{ (mod 2)}}}^{n-4} \tilde{h}_i(n-2)\prod^{n-4}_{\substack{ l =-(n-4) \\ l \neq i}} \Big( \frac{\lambda-l}{i-l}\Big) \\
\tilde{h}_{n}(\lambda)&:=&\sum_{\substack{i=-(n-2) \\ i \equiv n-2 \text{ (mod 2)}}}^{n-2} \tilde{h}_i(n)\prod^{n-2}_{\substack{ l =-(n-2) \\ l \neq i}} \Big( \frac{\lambda-l}{i-l}\Big),
\end{eqnarray*}
and $\tilde{h}_{-k}(\lambda):=\tilde{h}_k(-\lambda)$.
Then $h \in \prescript{}{n}{\mathcal{A}}_n$.
\end{proof}

We observe that an anologue of
Thm.~\ref{prop:Amm}, in particular relation (\ref{eq:Ammflat}), is also true in general for distinct $K$-type $n$ and $m$.
\begin{cor} \label{cor:Anmflat}
With the notations above, $\prescript{}{n}{\mathcal{A}}_m \cong \emph \Hol(\mathbb{C}) \otimes_{\emph \Pol(\mathbb{C})} \prescript{}{n}{\emph \Pol}_m.$
\end{cor}

\begin{proof}
As a consequence of Thm.~\ref{thm:Meta3SL2C}, the 
$\text{Hol}(\mathbb{C})$-module $\prescript{}{n}{\mathcal{A}}_m$ is isomorphic to $\prescript{}{l}{\mathcal{A}}_l$, and $\prescript{}{n}{\text{Pol}}_m$ is isomorphic to $\prescript{}{l}{\text{Pol}}_l$ as $\text{Pol}(\mathbb{C})$-module. Then, we can apply Thm.~\ref{prop:Amm} and obtain the desired result.
\end{proof}

\subsection*{Acknowledgement}
This research is supported by the research project PRIDE15/10949314/GSM
of the Luxembourg National Research Fund and by the University of Luxembourg.

\begin{minipage}[t][2.5cm][b]{0.7\textwidth}
Université du Luxembourg,\\
Faculty of Science, Technology and Medicine,\\
Department of Mathematics \\
Email addresses: \href{mailto:guendalina.palmirotta_@_uni.lu}{guendalina.palmirotta@uni.lu}
\end{minipage}

\begin{minipage}[t][2.5cm][b]{0.7\textwidth}
Université du Luxembourg,\\
Faculty of Science, Technology and Medicine,\\
Department of Mathematics \\
Email address: \href{mailto:martin.olbrich_@_uni.lu}{martin.olbrich@uni.lu}
\end{minipage}

\end{document}